\documentclass[12pt,a4paper]{amsart}
\usepackage{amssymb}
\usepackage{amsfonts}
\usepackage{amsthm}
\usepackage{amsmath}
\usepackage{amscd}
\usepackage[latin2]{inputenc}
\usepackage{t1enc}
\usepackage[mathscr]{eucal}
\usepackage{indentfirst}
\usepackage{graphicx}
\usepackage{graphics}
\numberwithin{equation}{section}
\usepackage[margin=2.9cm]{geometry}
\usepackage{epstopdf} 
\usepackage{xcolor}
 
\usepackage{verbatim}
 \usepackage{enumitem}
\usepackage{tikz}
\usepackage[percent]{overpic}
\usepackage{xcolor} 

\theoremstyle{plain}
\newtheorem{theorem}{Theorem}[section]
\newtheorem{lemma}[theorem]{Lemma}
\newtheorem{corollary}[theorem]{Corollary}
\newtheorem{proposition}[theorem]{Proposition}

 \theoremstyle{definition}
 \newtheorem{condition}[theorem]{Condition}

\newtheorem{remark}[theorem]{Remark}
\newtheorem{?}[theorem]{Problem}

\newcommand{\R}{\mathbb{R}}

\newcommand{\coup}{{g}}

\newcommand{\bx}{{x}}

\title[Quantum field theory and inverse problems]{Quantum field theory and inverse problems: Imaging with Entangled Photons}

\author{Matti Lassas}
\address{Department of Mathematics and Statistics, University of Helsinki, Helsinki, Finland}
\email{matti.lassas@helsinki.fi}

\author{Medet Nursultanov}
\address{Department of Mathematics and Statistics, University of Helsinki, Helsinki, Finland}
\email{medet.nursultanov@gmail.com}

\author{Lauri Oksanen}
\address{Department of Mathematics and Statistics, University of Helsinki, Helsinki, Finland}
\email{lauri.oksanen@helsinki.fi}

\author{John C. Schotland }
\address{Department of Mathematics and Department of Physics, Yale University, New Haven, CT, USA}
\email{john.schotland@yale.edu}

\date{}

\begin{document}



\begin{abstract}
We consider the quantum field theory for a scalar model of the electromagnetic field interacting
with a system of two-level atoms. In this setting, we show that it is possible to uniquely determine the density of atoms from measurements of the source to solution map for a system of nonlocal partial differential equations, which describe the scattering of a two-photon state from the atoms. The required measurements involve correlating the outputs of a point detector with an integrating detector, thereby exploiting information about the entanglement of the photons.
\end{abstract}

\maketitle

\section{Introduction}

Scattering experiments are among the most powerful tools to probe the structure of matter. Such experiments play a fundamental role in nearly all branches of physics. Mathematically, their analysis is framed as an inverse problem. In quantum mechanics, for instance, this consists of determining the potential in the Schr\"odinger equation from suitable measurements, a problem that is relatively well understood. 
However, there is a large class of physical phenomena that are not described in this language. These appear in physical settings where the number of particles is not conserved. For example, in quantum field theory, field excitations (particles) can be created and destroyed. Similar considerations appear in condensed matter physics, where quasiparticles are likewise created and destroyed. 

In this paper we consider an inverse problem in quantum field theory that arises in the quantum electrodynamics of a system of two-level atoms.
The inverse problem consists of recovering the density of atoms from measurements of the source to solution map for a system of nonlocal partial differential equations. Physically, this describes the scattering of an \emph{entangled} two-photon state from the atoms.
We show that it is possible to uniquely determine the density from such measurements. We emphasize that entanglement of the photons is a necessary condition for this result to hold.
To the best of our knowledge, this work reports the first study of an inverse scattering problem in quantum field theory. 
 
\subsection{Model}
Let $n\geq 2$ and $\rho \in C_0^\infty(\mathbb{R}^n)$ be a non-negative function, and define the compact set $\Sigma = \operatorname{supp}(\rho) \subset \mathbb{R}^n$. We consider the following system of nonlocal partial differential equations:
\begin{align}
\label{eq:dynamics1}
\nonumber i\partial_t\psi_2(t,x_1,x_2) 
&= (-\Delta_{x_1})^{1/2}\psi_2(t,x_1,x_2)+(-\Delta_{x_2})^{1/2}\psi_2(t,x_1,x_2)\\
\nonumber &\quad +\frac{g}{2}\left(\rho(x_1)\psi_1(t,x_1,x_2)+\rho(x_2)\psi_1(t,x_2,x_1)\right), \\
i\partial_t \psi_1(t,x_1,x_2)
&= \left[(-\Delta_{x_2})^{1/2}+\Omega\right]\psi_1(t,x_1,x_2) + 2g\psi_2(t,x_1,x_2) \\
\nonumber &\quad - 2g\rho(x_2)a(t,x_1,x_2), \\
i\partial_t a(t,x_1,x_2) 
\nonumber&= 2\Omega\, a(t,x_1,x_2)+ \frac{g}{2}\left(\psi_1(t,x_2,x_1)-\psi_1(t,x_1,x_2)\right),
\end{align}
together with the initial conditions
\begin{equation}\label{initial_cond}
    \psi_1 \arrowvert_{t=0}=0, \qquad
    \psi_2 \arrowvert_{t=0}=f_0, \qquad
    a \arrowvert_{t=0}=0,
\end{equation}
where $f_0 \in C_0^\infty(\mathbb{R}^{n}\times \mathbb{R}^{n})$ is symmetric, that is, $f_0(x_1,x_2) = f_0(x_2,x_1)$. Eq.~\eqref{eq:dynamics1} has its origins in the quantum electrodynamics of two-level atoms, as introduced in~\cite{KS1,KS2} and further developed in~\cite{HKSW1,HKSW2,HKRS,KSS}. Here $\psi_2(t,x_1,x_2)$ denotes the probability amplitude for creating two photons at the points $x_1$ and $x_2$ at time $t$, $\psi_1(t,x_1,x_2)$ is the amplitude for exciting an atom at $x_1$ and creating a photon at $x_2$, $a(t,x_1,x_2)$ is the amplitude for exciting two atoms at $x_1$ and $x_2$, and $\rho$ is the number density of the atoms. The constants $g$ and $\Omega$ are nonnegative and correspond to the atom-photon coupling and the atomic transition frequency, respectively. 

In this setting, the model describes the scattering of a two-photon state from the atoms. If $\psi_2$ is not factorizable as a function of $x_1$ and $x_2$, this corresponds to an entangled two-photon state. We emphasize that while the derivation of \eqref{eq:dynamics1} is carried out within the framework of quantum field theory, the results presented in this work are mathematically rigorously justified.

\subsection{Inverse Problem}
To reformulate the system \eqref{eq:dynamics1} in a compact form, we introduce the following notation:
\begin{equation*}
    \tilde{\psi}_1(t,x_1,x_2) = \psi_1(t,x_2,x_1), \qquad
    \rho_1(x_1,x_2) = \rho(x_1), \qquad
    \rho_2(x_1,x_2) = \rho(x_2).
\end{equation*}
In addition, we define the matrix of operators $\mathcal{A}$ by
\begin{equation*}
	\begin{bmatrix}
		i\partial_t - (-\Delta_{x_1})^{\frac{1}{2}} - (-\Delta_{x_2})^{\frac{1}{2}} & -g \sqrt{\rho_2} & -g \sqrt{\rho_1} & 0 \\
		- g \sqrt{\rho_2} & i\partial_t - (-\Delta_{x_1})^{\frac{1}{2}} - \Omega & 0 & -g \sqrt{\rho_1} \\
		- g \sqrt{\rho_1} & 0 & i\partial_t - (-\Delta_{x_2})^{\frac{1}{2}} - \Omega & g \sqrt{\rho_2} \\
		0 & -g \sqrt{\rho_1} & g \sqrt{\rho_2} & i\partial_t - 2\Omega
	\end{bmatrix}
\end{equation*}
and vector-valued functions
\begin{equation}\label{G_psi}
	u = \sqrt{2}\left[ \psi_2, \tfrac{1}{2} \sqrt{\rho_2} \tilde{\psi}_1, \tfrac{1}{2} \sqrt{\rho_1} \psi_1, \sqrt{\rho_1\rho_2}a\right]^T,
	\qquad
	f = [f_0, 0, 0, 0]^T .
\end{equation}
Then, we obtain
\begin{equation}\label{main_eq}
    \begin{cases}
        \mathcal{A} u = 0,\quad \hbox{for }t>0,\\
        u\arrowvert_{t=0} = f.
    \end{cases}
\end{equation}
We study the above problem~\eqref{main_eq} for initial data of the form described above, that is,  we assume that $f$ belongs to the set
\begin{equation*}
	\mathcal{C}_{\operatorname{sym}}(S)= \left\{ h \in C_0^\infty( S; \mathbb{C}^4) : h_0(x_1, x_2) = h_0(x_2, x_1), \;  h_1 = h_2 = h_3 = 0 \right\},
\end{equation*}
where $h = [h_0, h_1, h_2, h_3]^T$. In Section \ref{Sec:direct_problem}, we show that for such $f$, the problem above admits a unique smooth solution, which we denote by $u^f = [u_0^f, u_1^f, u_2^f, u_3^f]^T$.

To define the inverse problem, we first describe the measurement setting; see Figure~\ref{fig:measurement}. Let $W_1, W_2 \subset \mathbb{R}^n$ be two open, non-empty sets representing spatial regions where the two photons may be detected. We assume that measurements are spatially resolved with respect  to the first spatial variable, meaning that we have pointwise information in $x_1 \in W_1$. In contrast, measurements in the second variable are limited. We only have access to integrated (averaged) information for $x_2 \in W_2$.

\begin{figure}[t]
    \centering
    \begin{overpic}[width=1\textwidth]{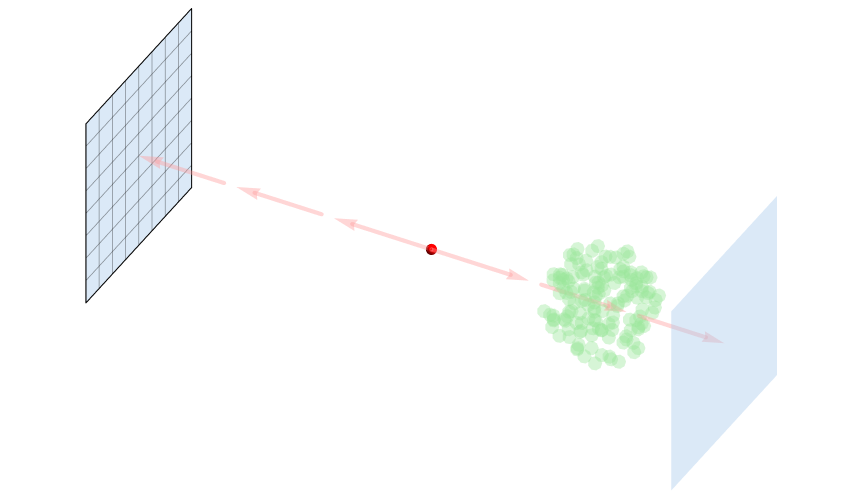}
        \put(0,18){Spatially resolved detector}
        \put(68,-3){Integrating detector }
        \put(46,26){Source}
        \put(65,31){Medium}
    \end{overpic}
    \bigskip
    \caption{Schematic illustration of the measurement apparatus. One photon is 
    detected by a spatially-resolved detector, while the second photon, after 
    interacting with the medium, is recorded by an integrating detector. The correlation of  the detector outputs defines the measurement $\Lambda f$.}
    \label{fig:measurement}
\end{figure}

To formalize the above, we fix a non-negative function $\chi \in C_0^\infty(\mathbb{R}^n)$ with $\operatorname{supp}(\chi) = \overline{W_2}$, and define the measurement operator
\begin{align}
\begin{split}
\label{measurement}
	\Lambda :   \mathcal{C}_{\operatorname{sym}}(S) &\to C^\infty ((0,\infty)\times W_1),\\
	\Lambda f(t,x_1) & = \int_{\mathbb R^n} \chi(x_2) \left|u_0^f(t, x_1, x_2)\right|^2 \, dx_2.
\end{split}
\end{align}
The inverse problem is to determine the unknown density $\rho$ from knowledge of $\Lambda$.

\subsection{Main Result}

To extract information about the unknown density $\rho$ from the data $\Lambda$, we require a specific geometric configuration of the source region $S$, the measurement domains $W_1$ and $W_2$, and the support of the density $\Sigma$. First, we need all relevant photon paths to  reach the medium. To ensure this, we assume that there exists a source point $z = (z_1, z_2) \in S$ such that every point of the support $\Sigma$ lies on a segment from $z_2$ to some point in $W_2$. This guarantees that the rays associated with the second photon's path probe the entire support of $\rho$.
Second, for each such direction, we require a ray of equal length from a point $x_1 \in W_1$ to $z_1$ that avoids $\Sigma$. These rays serve as reference paths that allow us to isolate the contribution of $\rho$. By properly adjusting the propagation times and exploiting the symmetry of the equations, we can access integrals of $\rho$ along the segment from $z_2$ to $x_2$. The informal geometric description above is made precise in Condition~\ref{condition_intro}, which contains the full set of assumptions used in the proof.

Under these geometric conditions, our main result is as follows.

\begin{theorem}
\label{main result}
	Assume that Condition \ref{condition_intro} holds. Then the measurement map $\Lambda$ determines the density $\rho$ uniquely.
\end{theorem}

For a related inverse problem in which the system is driven by a time-dependent source term rather than by initial data, the geometric assumptions can be relaxed. We analyze this variant in Appendix~\ref{appendix: source problem} and prove an analogous uniqueness result under weaker hypotheses.

We empahsize that in the proof of Theorem \ref{main result} we use
non-factorizable initial values $f(x_1,x_2)$, that is, functions $f(x_1,x_2)$ that can not be written
in the form $f_1(x_1)f_2(x_2)$.  Physically, this means that to solve the inverse 
problem of determining the density $\rho(x)$  we use initial
values $f$ that produce
two entangled particles, see Remark \ref{Remark 1 on entanglement} below.

\subsection{Physical Considerations and Applications}
The physical setting we consider is a quantum field theory that describes the interaction of light and matter. This theory is based on several assumptions \cite{KS1}. We employ a scalar model of the electromagnetic field, so that effects due to polarization are not considered. Matter is taken to consist of two-level atoms. We also make use of the dipole and rotating wave approximations, which are standard in quantum optics \cite{MW}. The former holds when the wavelength of light is large compared to the atomic size. The latter breaks down in the strong-coupling regime, where the optical frequency is comparable to or greater than the atom-field coupling. In addition, all fields are quantized in real space \cite{KS1,KS2}. This leads to a formulation of quantum electrodynamics in which the probability amplitudes that define a two-excitation quantum state obey the nonlocal partial differential equations \eqref{eq:dynamics1}. The corresponding physical consequences have been explored in a series of papers that address many-body problems in quantum optics, especially the derivation of kinetic equations for random media, the characterization of bound states and resonances, and band structure of periodic atomic systems \cite{KS1,KS2,HKSW1,HKSW2,HKRS,KSS}. 

This paper is the first to consider the inverse problem for  \eqref{eq:dynamics1}. A potential application is to optical imaging with quantum states of light. The key idea is to exploit the quantum nature of light to achieve capabilities, such as increased resolution and improved signal-to-noise performance, that surpass the limits imposed by classical physics \cite{Defienne,MTGP}. The experiment shown in Fig.~\ref{fig:geometry} is closely related to ghost imaging and imaging with undetected photons \cite{Defienne}. These methods utilize quantum correlations between pairs of entangled photons to form an image of an object. Here one photon from the pair interacts with the object and is recorded by an integrating detector, while the second photon does not interact with the object and is registered by a spatially-resolved detector. The correlations of the detector outputs, viewed as an image, is a direct visualization of the function $\Lambda f$ defined by \eqref{measurement}, and corresponds to the observable in a ghost imaging experiment. 

\section{Direct problem}\label{Sec:direct_problem}
In this section, we study the direct problem. We will show that equation \eqref{main_eq}, with a suitable initial condition, has a unique smooth solution.

We denote 
\begin{equation*}
    \begin{array}{ll}
        L = (-\Delta_{x_1})^{1/2} + (-\Delta_{x_2})^{1/2},\\
        L_1 = (-\Delta_{x_1})^{1/2},\\
        L_2 = (-\Delta_{x_2})^{1/2},\\
        L_3 = 0,
    \end{array}
\qquad
    \begin{array}{ll}
         D = i\partial_t - ((-\Delta_{x_1})^{1/2} + (-\Delta_{x_2})^{1/2}), \\
         D_1 = i\partial_t - (-\Delta_{x_1})^{1/2}, \\
         D_2 = i\partial_t - (-\Delta_{x_2})^{1/2},\\
         D_3 = i\partial_t.
    \end{array}
\end{equation*}

Let us set
\begin{equation*}
    \mathcal{L}= 
  \begin{bmatrix}
        L & 0 & 0 & 0\\
        0 & L_1 & 0 & 0\\
        0 & 0 & L_2 & 0\\
        0 & 0 & 0 & L_3
    \end{bmatrix},
    \qquad
    \mathcal{D}= 
  \begin{bmatrix}
        D & 0 & 0 & 0\\
        0 & D_1 & 0 & 0\\
        0 & 0 & D_2 & 0\\
        0 & 0 & 0 & D_3
    \end{bmatrix},
\end{equation*}
and
\begin{equation}\label{B}
    \mathcal{B}=
    \begin{bmatrix}
        0 & g \sqrt{\rho_2} & g \sqrt{\rho_1} & 0\\
        g \sqrt{\rho_2} & \Omega & 0 &  g \sqrt{\rho_1}\\
       g \sqrt{\rho_1} & 0 & \Omega & - g \sqrt{\rho_2}\\
       0 &  g \sqrt{\rho_1} &  - g \sqrt{\rho_2} & 2\Omega
    \end{bmatrix}.
\end{equation}
Using this notation, we rewrite equation \eqref{main_eq} as follows:
\begin{equation*}
    \begin{cases}
        \partial_t u =  - i(\mathcal{L} + \mathcal{B})u,\\
        u\arrowvert_{t=0} = f.
    \end{cases}
\end{equation*}
We will consider $\mathcal{P} = -i(\mathcal{L} + \mathcal{B})$ as an operator on the Hilbert space $H^k(\mathbb{R}^{2n}; \mathbb{C}^4)$. Hence, we need to define the operators $-\Delta_{x_1}$ and $-\Delta_{x_2}$ on the Hilbert space $H^k(\mathbb{R}^{2n})$ to act as Laplace operators with respect to the variables $x_1$ and $x_2$ separately. This can be done via a sesquilinear form. We set 
\begin{equation*}
    H^{(k;1,0)}(\mathbb{R}^{2n})  = \left\{ h\in H^k(\mathbb{R}_{x_1}^n\times \mathbb{R}_{x_2}^n): \nabla_{x_1} h\in H^k(\mathbb{R}^{2n}; \mathbb{C}^n) \right\}.
\end{equation*}
Consider the following sesquilinear form in the Hilbert space $H^k(\mathbb{R}^{2n})$:
\begin{equation*}
    a_1^k(u,v) = (\nabla_{x_1}u,\nabla_{x_1}v)_{H^k(\mathbb{R}^{2n})},
    \qquad
    D(a_1^k) = H^{(k;1,0)}(\mathbb{R}^{2n}).
\end{equation*}
The form is non-negative, densely defined, symmetric, and closed in $H^k(\mathbb{R}^{2n})$. Therefore, it generates a non-negative self-adjoint operator in $H^k(\mathbb{R}^{2n})$, see Theorem 2.1 in \cite[Ch. 6]{Kato1995}. We denote this operator by $-\Delta_{1,k}$. Since $-\Delta_{1,k}$ is non-negative, $L_1^k=(-\Delta_{1,k})^{1/2}$ is well defined. Moreover, by Theorem 2.23 in \cite[Ch. 6]{Kato1995}, we know that $D(L_1^k) = H^{(k;1,0)}(\mathbb{R}^{2n})$. Similarly, we set 
\begin{equation*}
    H^{(k;0,1)}(\mathbb{R}^{2n}) = \left\{ h\in H^k(\mathbb{R}^{2n}): \nabla_{x_2} h\in H^k(\mathbb{R}^{2n}; \mathbb{C}^n) \right\}.
\end{equation*}
and define a non-negative self-adjoint operator $L_2^k=(-\Delta_{2,k})^{1/2}$ in $H^k(\mathbb{R}^{2n})$ with domain $D(L_2^k) = H^{(k;0,1)}(\mathbb{R}^{2n})$. Finally, we define
\begin{equation*}
    L^k = L_1^k + L_2^k,
\end{equation*}
with the domain 
\begin{equation*}
    D(L^k) = D(L_1^k) \cap D(L_2^k) = H^{k+1}(\mathbb{R}^{2n}).
\end{equation*}
Next, we will show that $L^k$ is self-adjoint. Define 
\begin{align*}
    U: H^k(\mathbb{R}^{2n}) &\rightarrow L^2(\mathbb{R}^{2n}),\\
    Uf & = (1 + |\cdot|^2)^\frac{k}{2} \widehat{f},
\end{align*}
where $\widehat{f}$ is the Fourier transform of $f$. Then $\|Uf\|_{L^2}=\|f\|_{H^k}$, and for any
$h\in L^2$ the choice $\widehat{f}(\xi)=(1+|\xi|^2)^{-k/2}h(\xi)$ gives $Uf=h$; hence $U$
is a unitary operator.

For $m_j(\xi) = |\xi_{x_j}| $, $j=1,2$, consider the multiplication operator
\begin{align*}
    M_j: D(M_j) \subset L^2(\mathbb{R}^{2n}) &\rightarrow L^2(\mathbb{R}^{2n}),\\
    M_jf & = m_j f,
\end{align*}
where
\begin{equation*}
    D(M_j) = \{f\in L^2(\mathbb{R}^{2n}): \; m_jf \in L^2(\mathbb{R}^{2n})\}.
\end{equation*}
Note that 
\begin{multline*}
    D(L_j^k) = \{f \in L^2(\mathbb{R}^{2n}): \; (1 + |\cdot|^2)^\frac{k}{2} \widehat{f} \in L^2(\mathbb{R}^{2n}) \text{ and }    m_j Uf \in L^2(\mathbb{R}^{2n})\}
    \\ =\{f \in L^2(\mathbb{R}^{2n}): \;  Uf \in D(M_j)\} = U^{-1} D(M_j).
\end{multline*}
Therefore, it is straightforward to check that $L_j^k = U^{-1}M_jU$, for $j=1,2$.

We set 
\begin{align*}
    &M : D(M) \subset L^2(\mathbb{R}^{2n}) \rightarrow L^2(\mathbb{R}^{2n}),\\
    &Mf  = (m_1 + m_2) f,
\end{align*}
where
\begin{equation*}
    D(M) = \{f\in L^2(\mathbb{R}^{2n}): \; (m_1 + m_2) f \in L^2(\mathbb{R}^{2n})\}.
\end{equation*}
Since $m_1$ and $m_2$ are non-negative, we obtain
\begin{equation*}
    D(M) = D(M_1) \cap D(M_2),
\end{equation*}
and hence, $M=M_1 + M_2$. Therefore, 
\begin{equation*}
    L^k = L_1^k + L_2^k = U^{-1}(M_1+M_2) U = U^{-1}M U.
\end{equation*}
Moreover, since $m_1+m_2$ is a real-valued function, it follows from Proposition 1 in \cite[Ch. VIII.3 ]{ReedSimon1} that $M$ is self-adjoint, and consequently, $L^k$ is self-adjoint.

\begin{lemma}\label{direct_problem_est}
    Let $k\in \mathbb{N}$ and assume that $f\in H^k(\mathbb{R}^{2n};\mathbb{C}^4)$. Then equation \eqref{main_eq} has a unique solution $u^f$ such that  $u^f(t,\cdot)\in H^k(\mathbb{R}^{2n};\mathbb{C}^4)$ for $t>0$. Moreover, for any compact set $K\subset [0,\infty)$, there exists a constant $C_K>0$ such that 
    \begin{equation*}
        \sup_{t\in K} \|u^f(t,\cdot)\|_{H^k(\mathbb{R}^{2n};\mathbb{C}^4)} \leq C_K \|f\|_{H^k(\mathbb{R}^{2n};\mathbb{C}^4)}.
    \end{equation*}
\end{lemma}
\begin{proof}
    Let $\mathcal{B}_k$ be the bounded operator on $H^k(\mathbb{R}^{2n}; \mathbb{C}^4)$ given by \eqref{B} and $\mathcal{L}_k$ be the operator on $H^k(\mathbb{R}^{2n}; \mathbb{C}^4)$ given by
\begin{equation*}
   \mathcal{L}_k =
   \begin{bmatrix}
        L^k & 0 & 0 & 0\\
        0 & L_1^k  & 0 & 0\\
       0 & 0 & L_2^k  & 0\\
       0 & 0 &  0 & 0
\end{bmatrix}
\quad
D(\mathcal{L}_k) = 
\begin{bmatrix}
H^{k+1}(\mathbb{R}^{2n}) \\
H^{(k;1,0)}(\mathbb{R}^{2n})\\
H^{(k;0,1)}(\mathbb{R}^{2n})\\
H^{k}(\mathbb{R}^{2n})
\end{bmatrix},
\end{equation*}
Consider a vector-valued ordinary differential equation
\begin{equation}\label{system}
    \begin{cases}
        \partial_t u = (-i(\mathcal{L}_k + \mathcal{B}_k))u,\\
        u\arrowvert_{t=0} = f.\hspace{20mm}
    \end{cases}
\end{equation}
Since $L^k$, $L_1^k$, and $L_2^k$ are self-adjoint, we conclude that $\mathcal{L}_k$ is also self-adjoint. Therefore, $-i\mathcal{L}_k$ is a skew-adjoint operator and hence, by the Hille-Yosida theorem, generates a strongly continuous semigroup; see for example Theorem 3.5 in \cite[Ch. 2]{EngelNagel}. Moreover, since $\rho$ is a compactly supported smooth function, we know that $-i\mathcal{B}_k$ is a bounded operator. Therefore, by the bounded perturbation theorem, the operator $\mathcal{P}_k = -i(\mathcal{L}_k + \mathcal{B}_k)$ generates a strongly continuous semigroup $\{T_k(t)\}_{t\geq 0}$; see for example Theorem 1.3 in \cite[Ch. 3]{EngelNagel}. In particular, the system \eqref{system} has a unique solution $u^f= T_k(t) f$. The required estimate follows from the local boundedness of $\{T_k(t)\}_{t>0}$; see for example Theorem 1.3 in \cite[Ch. 3]{EngelNagel}.
\end{proof}

We will need the following lemma.
\begin{lemma}
    Assume that $f\in H^k(\mathbb{R}^{2n};\mathbb{C}^4)$ for all $k\in \mathbb{N}$. Then equation \eqref{main_eq} has a unique smooth solution.
\end{lemma}
\begin{proof}
    From the proof of the previous result, we know that $u_k^f= T_k(t) f$ solves \eqref{system}. Since $f\in H^k(\mathbb{R}^{2n};\mathbb{C}^4)$ for all $k\in \mathbb{N}$, the uniqness of the solution implies that all $\{u_k^f\}$ are the same function, which we denote by $u^f$.

    It remains to show that $u^f$ is smooth. Let $\alpha$ be a multi-index and $l\in \mathbb{N}$. For sufficiently large $k$, we write
\begin{equation*}
    \partial_x^\alpha \partial_t^l u^f = \partial_x^\alpha T_k(t) \mathcal{P}_k^l f.
\end{equation*}
For sufficiently large $k$, using strong continuity of the semigroup on $H^k$, together with the boundedness of $\partial_x^\alpha: H^k\to H^{k-|\alpha|}$ and the Sobolev embedding theorem, we conclude that $\partial_x^\alpha\partial_t^{\,l}u^f$ is continuous on $(0,\infty)\times\mathbb{R}^{2n}$.

\end{proof}

\subsection{Direct problem for auxiliary equation} Here, we solve the direct problem for the following equation
\begin{equation}\label{auxiliary_eq}
    \begin{cases}
        \mathcal{A}u = F, \\
        u\arrowvert_{t<0}  = 0,
    \end{cases}
\end{equation}
for $F$ being supported in $(0,\infty)\times\mathbb{R}^{2n}$.

\begin{proposition}\label{cor: stability_for_aux_eq}
    Let $k \in \mathbb{N}$ sufficiently large and $F \in C((0, \infty); H^{k+1}(\mathbb{R}^{2n}; \mathbb{C}^4))$ be such that $F\arrowvert_{t<\varepsilon} = 0$ for some $\varepsilon>0$. Then, equation \eqref{auxiliary_eq} has a unique solution $u^F \in C^1(\mathbb{R}; H^{k+1}(\mathbb{R}^{2n}; \mathbb{C}^4))$. Moreover, for any $T>0$, there exists $C_T>0$ such that 
    \begin{equation*}
        \sup_{t\in [0,T]} \|u^F(t,\cdot)\|_{H^{k}(\mathbb{R}^{2n};\mathbb{C}^4)} \leq C_T \sup_{t\in [0,T]} \|F(t,\cdot) \|_{H^{k}(\mathbb{R}^{2n};\mathbb{C}^4)}.
    \end{equation*}
    Finally, if $F \in C^\infty(\mathbb{R}^{2n+1})\cap C \big((0, \infty); H^k(\mathbb{R}^{2n}; \mathbb{C}^4)\big)$ for all $k \in \mathbb{N}$, then $u^F$ is smooth. 
\end{proposition}

\begin{proof}
    Let $\{T_k(t)\}_{t \geq 0}$ denote the semigroup introduced in the proof of Lemma \ref{direct_problem_est}, and let $\mathcal{P}_k$ be its generator. We note that $F(t) \in D(\mathcal{P}_k)$ for all $t>0$ and $\mathcal{P}_k$ is bounded from $H^{k+1}(\mathbb{R}^{2n}; \mathbb{C}^4)$ into $H^k(\mathbb{R}^{2n}; \mathbb{C}^4)$. Therefore, Corollary 7.8 in \cite{EngelNagel} implies that
    \begin{equation}\label{eq: solution}
        u^F(t,x) = \int_0^t T_k(t - s)F(s,x) ds
    \end{equation}
    is the unique solution to \eqref{auxiliary_eq} and $u^F \in C^1(\mathbb{R}; H^{k+1}(\mathbb{R}^{2n}; \mathbb{C}^4))$. Applying the triangle inequality to the right-hand side of \eqref{eq: solution}, together with local boundedness of $T_k(t)$, yields the desired estimate.

    Let us prove the last statement. Since $F \in C \big((0, \infty); H^k(\mathbb{R}^{2n}; \mathbb{C}^4)\big)$ for all $k \in \mathbb{N}$, we can define, for all $m\in \mathbb{N}$ and multi-index $\alpha$, the function 
    \begin{equation*}
        g_{m,\alpha}(s,t,x) = \partial_x^\alpha \partial_t^m T_k(t-s)F(s,x).
    \end{equation*}
    Due to the properties of the semigroup, it follows that
    \begin{equation*}
        g_{m,\alpha}(s,t,x) = \partial_x^\alpha  T_k(t-s) \mathcal{P}_k^m F(s,x).
    \end{equation*}
    Using strong continuity of the semigroup on $H^k$, together with the boundedness of $\partial_x^\alpha: H^k\to H^{k-|\alpha|}$ and the Sobolev embedding theorem, we conclude that $g_{m,\alpha}$ is a continuous on function. This enables us to apply the Leibniz rule and conclude that $\partial_t^m \partial_x^\alpha u^F$ can be written as a linear combination of terms of the form
    \begin{equation*}
        \partial_x^\beta \partial_t^k \mathcal{P}_k^l F(t,x),
        \qquad
         \int_0^t g_{k,\beta} (s,t,x) ds.
    \end{equation*} 
    Since $F$ is smooth, the first term is also smooth. The second term is continuous due to the previously established continuity of $g_{m, \beta}$. Hence $u^F$ is a smooth function.
\end{proof}

\section{Microlocal analysis}
In this section, we study a solution to equation \eqref{auxiliary_eq}, where $f$ is a conormal distribution to be specified later. We note that $\mathcal{A}$ is not a pseudodifferential operator because it includes terms such as 
\begin{equation*}
    D = i\partial_t  - ((-\Delta_{x_1})^{1/2} + (-\Delta_{x_2})^{1/2}),
\end{equation*}
whose symbol is not smooth and does not have the required decay. However, in certain cases, if $F$ is supported outside of the singularities of the symbol of $\mathcal{A}$, then we can treat $\mathcal{A}$ as a pseudodifferential operator to show that a solution is also a conormal distribution.

To eliminate the singularity of the symbol of $\mathcal{A}$, we introduce several cutoff functions. Let $\chi_0$ be a smooth function on the unit cosphere bundle $S^*\mathbb{R}^{2n+1}$ such that 
\begin{equation*}
    \chi_0 (t,x_1,x_2, \tau, \xi_1,\xi_2) = 
    \begin{cases}
        1 & \text{if } \min(|\xi_1|, |\xi_2|) >\varepsilon,\\
        0 & \text{if } \min(|\xi_1|, |\xi_2|) <\varepsilon/2,
    \end{cases}
\end{equation*}
for some small $\varepsilon>0$. Then, we extend this function to a $0$-homogeneous function $\chi_1$:
\begin{equation}
    \chi_1 (t,x_1,x_2, \tau, \xi_1,\xi_2) = \chi_0 \left(t,x_1,x_2, \frac{\tau}{|(\tau,\xi_1,\xi_2)|}, \frac{\xi_1}{|(\tau,\xi_1,\xi_2)|},\frac{\xi_2}{|(\tau,\xi_1,\xi_2)|}\right).
\end{equation}
To remove the singularity at the origin, we choose a smooth function $\chi_2$ such that 
\begin{equation*}
    \chi_2 (t,x_1,x_2, \tau, \xi_1,\xi_2) = 
    \begin{cases}
        0 & \text{if } |(\tau, \xi_1, \xi_2)| >\varepsilon,\\
        1 & \text{if } |(\tau, \xi_1, \xi_2)| <\varepsilon/2,
    \end{cases}
\end{equation*}
for some small $\varepsilon>0$. Finally, we define
\begin{equation}\label{chi}
   \chi_3 = (1 - \chi_2) \chi_1. 
\end{equation}
Let $X$ be a pseudodifferential operator with symbol $\chi_3$. Then, the composition
\begin{equation*}
    D_{X} = D X
\end{equation*}
is a pseudodifferential operator with the principal symbol
\begin{equation*}
    \sigma[D_{X}](t,x_1,x_2,\tau,\xi_1,\xi_2) = (\tau - |\xi_1| - |\xi_2|) \chi_3(\tau,\xi_1,\xi_2).
\end{equation*}
Consequently, the characteristic set is
\begin{equation*}
    \mathrm{Char}(D_{X}) = \{(t,x_1,x_2,\tau,\xi_1,\xi_2)\in T^*\mathbb{R}^{2n+1}: \tau = |\xi_1| + |\xi_2|\}.
\end{equation*}
In the region where $\chi_3 = 1$, the Hamiltonian vector field associated to $\sigma[D_{X}]$ is given by 
\begin{equation*}
    H = \partial_t - \frac{\xi_1}{|\xi_1|}\partial_{x_1} - \frac{\xi_2}{|\xi_2|}\partial_{x_2}.
\end{equation*}
The characteristic curve which passes through the point $(t^*, x_1^*,x_2^*,\tau^*,\xi_1^*,\xi_2^*)$ is given by 
\begin{align*}
    \beta(t) = \left(t + t^*, -\frac{\xi_1^*}{|\xi_1^*|}t + x_1^*,-\frac{\xi_2^*}{|\xi_2^*|}t + x_2^*, \tau^*, \xi_1^*,\xi_2^* \right).
\end{align*}
Consider the cone given by the projections of the bicharacteristics through the origin
    \begin{align}
C = \{ (t, t  \kappa_1,t  \kappa_2) \mid t \in \R,\ \kappa_1, \kappa_2 \in \mathbb{S}^{n-1} \} \subset \mathbb{R}_t\times \mathbb{R}^n_{x_1}\times \mathbb{R}^n_{x_2}.
    \end{align} 
Further, let $\varepsilon > 0$ and define the smooth manifold
    \begin{align}
K_\varepsilon = C \cap \{t > \varepsilon\}.
    \end{align}
The conormal bundle of $K_\varepsilon$ is 
\begin{equation*}
    N^*K_\varepsilon = \left\{ \left(t, t  \kappa_1, t  \kappa_2, \lambda_1 + \lambda_2, -\lambda_1  \kappa_1, - \lambda_2 \kappa_2\right): \;
        t > \varepsilon, \; \kappa_1, \kappa_2 \in \mathbb{S}^{n-1}, \; \lambda_1,\lambda_2 \in \mathbb{R} \right\}.
\end{equation*}
We note that 
\begin{multline*}
    N^*K_\varepsilon \cap \mathrm{Char}(D_{X})  \\
    = \left\{ \left(t, t  \kappa_1, t  \kappa_2, \lambda_1 + \lambda_2, -\lambda_1  \kappa_1, - \lambda_2 \kappa_2\right): \;
        t > \varepsilon, \; \kappa_1, \kappa_2 \in \mathbb{S}^{n-1}, \; \lambda_1,\lambda_2 \geq 0 \right\}.
\end{multline*}
In order to stay away from the singularity $\xi_1 = 0$ and $\xi_2=0$, we define, for $\varepsilon>0$, 
\begin{equation*}
    N_\varepsilon^* K_\varepsilon = \left\{ \left(t, t \kappa_1, t \kappa_2, \lambda_1 + \lambda_2, -\lambda_1 \kappa_1, - \lambda_2 \kappa_2\right): \;
        t > \varepsilon, \; \kappa_1, \kappa_2 \in \mathbb{S}^{n-1}, \; \lambda_1,\lambda_2 > \varepsilon \right\}.
\end{equation*}
The space of conormal distributions associated with $K_\varepsilon$ is denoted by $I(K_\varepsilon)$.

We will use the following notation.
\begin{equation*}
    U_c = \{(s,a\kappa_1,b\kappa_2): \; s\geq \varepsilon/c, \; a,b\in [0,c s], \text{ and } \kappa_1,\kappa_2\in \mathbb{S}^{n-1}\} \subset \mathbb{R}_{t} \times \mathbb{R}_{x_1}^n \times \mathbb{R}_{x_2}^n.
\end{equation*}
Any time slice of $U_c$ is compact in the spatial variables, since it is the product of two closed balls.

Now, we are ready to state and prove the main result of this section:
\begin{lemma}\label{par_main}
    Let $T>0$ and $F \in (I(K_\varepsilon))^4$ be such that
    \begin{equation*}
        \mathrm{supp}(F) \subset U_2\cap \{\varepsilon<t<T\}
        \qquad
        \mathrm{supp}(\sigma[F]) \subset N_\varepsilon^* K_\varepsilon
    \end{equation*}
    Assume $u = u^F$ solves equation \eqref{auxiliary_eq}. Then $u \in (I(K_\varepsilon))^4$ and $\sigma[u]\arrowvert_{N^*K_\varepsilon \setminus N_\varepsilon^* K_\varepsilon} =0$.
\end{lemma}
\begin{proof}
We set $\mathcal{A}_X = \mathcal{A}X$, $D_X = DX$ and $D_{j,X} = D_jX$, where $j=1,2,3$ and $X$ is the operator defined above. Let $v \in (I(K))^4$.  The principal symbol of $D_{X}$ vanishes on $N^* K$, and its subprincipal symbol vanishes everywhere. 
Thus, taking $X = \R^{2n+1}$, $Y$ a single point,
and the canonical relation corresponding to $N_\varepsilon^* K_\varepsilon \times Y$,
we get from \cite[Th. 25.2.4]{Hormander4} that
    \begin{align*}
\sigma[D_{X}v] = i^{-1} \mathcal L_H \sigma[v],
    \end{align*}
where $\mathcal L_H$ is the Lie derivative with respect to $H$.

We solve the equation
    \begin{align*}
\mathcal L_H \hat v_0^0 = i\sigma[F_0]
    \end{align*}
for a function $\hat v_0^0$ on $N_\varepsilon^* K_\varepsilon$
by integrating over the bicharacteristics 
    \begin{align*}
        \beta(t) = \left(t, t \kappa_1, t \kappa_2, \lambda_1 + \lambda_2, - \lambda_1  \kappa_1, - \lambda_2 \kappa_2\right),
    \end{align*}
with $\lambda_1$, $\lambda_2 > \varepsilon$ and $\kappa_1$, $\kappa_2 \in \mathbb{S}^{n-1}$,
starting from the initial condition 
    \begin{align*}
        \hat v_0^0\left(0, 0, 0, \lambda_1 + \lambda_2, -\lambda_1  \kappa_1, - \lambda_2 \kappa_2\right) = 0.
    \end{align*}
Note that $\hat v_0^0\arrowvert_{t<\varepsilon} = 0$ since $F\arrowvert_{t < \varepsilon} = 0$.

Note that $\hat v_0^0 = 0$ on $C \setminus K_\varepsilon$ since $F\arrowvert_{t < \varepsilon} = 0$. By extending $\hat v_0^0$ by zero to $N^*K_\varepsilon \setminus \operatorname{supp}(\sigma[F])$ and keeping $\operatorname{supp}(\hat v_0^0)\subset N_\varepsilon^*K_\varepsilon$, we obtain a symbol on $N^*K_\varepsilon$; see \cite[Remark 2, p. 69]{hormander1985}. Next, we note that 
\begin{equation*}
    \sigma[D_{k,X}] \neq 0,
    \qquad
    \text{on } 
    N_\varepsilon^* K_\varepsilon,
    \text{ for }
    k=1,2,3.
\end{equation*}
Therefore, there exists a symbol $\hat v_k^0$, which is of lower order than $F_k$, such that $\mathrm{supp} (\sigma[\hat v_k^0]) \subset N_\varepsilon^* K_\varepsilon$ and
\begin{equation*}
    \hat v_k^0 = \frac{\sigma[F_k]}{\sigma[D_{k,X}]}
    \qquad
    \text{on }
    N_\varepsilon^* K_\varepsilon, \text{ for }
    k=1,2,3.
\end{equation*}
We choose $v^0 \in (I(K))^4$ such that $\sigma[v^0] = (\hat v_0^0, \hat v_1^0,\hat v_2^0,\hat v_3^0)^t$. In particular, it follows that $\mathrm{supp} (\sigma[\hat v^0]) \subset N_\varepsilon^* K_\varepsilon$. We set $F^1 = F - \mathcal{A}_X v^0$. The maximal order among the components of $F^1$ is strictly lower than that of $F$. We continue by solving
\begin{align*}
    \mathcal L_H \hat v^j_0 = i\sigma[F^j], \qquad F^j = F^{j-1} - \mathcal{A}_X v^{j-1},
    \quad j\in \mathbb{N},
\end{align*}
and choosing $(\hat v_1^j, \hat v_2^j,\hat v_3^j)$ and $\hat v^j$ as above. Then, there is $\hat v = \sum_{j\in \mathbb{N}} \hat v^j$ in the sense of \cite[Prop. 1.8]{grigis1994}, and we choose $v \in (I(K_\varepsilon))^4$ such that $\mathrm{supp}(v) \subset U_3$ and $\sigma[v] = \hat v$.
We write $F^0 = F$ and $r_J = v - \sum_{j=0}^J v^j$. 
We have
\begin{equation*}
    \mathcal{A}_X v - F = \mathcal{A}_X r_J + \sum_{j=0}^J \mathcal{A}_X v^j - F^0 = \mathcal{A}_X r_J + \sum_{j=0}^J (F^j - F^{j+1}) - F^0 = \mathcal{A}_X r_J - F^{J+1},
\end{equation*}
and this can be made arbitrarily smooth by choosing $J$ large.
We conclude that 
    \begin{equation*}
        \mathcal{A}_Xv - F \in C^\infty(\R^{2n+1}).
    \end{equation*}
Writing $w = u - X v$, we obtain 
    \begin{equation*}
\mathcal{A}w = \mathcal{A}u - \mathcal{A}_Xv = F - \mathcal{A}_Xv \in C^\infty(\R^{2n+1}).
    \end{equation*}
We will prove that $w \in C^\infty(\mathbb{R}^{2n+1})$, which in turn implies that $u \in (I(K_\varepsilon))^4$ and $\sigma[u] = \sigma[v]$.

Let $\mu$, $\nu\in C^\infty(\mathbb{R}^{2n+1})$ be functions such that 
\begin{equation*}
    \mathrm{supp}(\nu) \subset U_4,
    \qquad
    \nu\arrowvert_{U_3} = 1,
\end{equation*}
and 
\begin{equation*}
    \mathrm{supp}(\mu) \subset U_6,
    \qquad
    \mu\arrowvert_{U_5} = 1,
\end{equation*}
Then, we express
\begin{align*}
    F - \mathcal{A}_Xv &= F - i\partial_tXv + \mathcal{L} Xv + \mathcal{B}Xv\\
    &=   F - i\partial_tXv + (1- \mu)\mathcal{L} X(\nu v) + \mu \mathcal{L} X(\nu v) +  \mathcal{B}Xv.
\end{align*}
Let us denote 
\begin{align*}
    & h_1 = (1- \mu)\mathcal{L} X(\nu v)\\
    & h_2 = F - i\partial_tXv + \mu \mathcal{L} X(\nu v) +  \mathcal{B}Xv.
\end{align*}
To examine $h_1$, we recall that $L_1 X$ is the pseudodifferential operator with the symbol 
\begin{equation*}
    \sigma[L_1 X] = |\xi_1| \chi_3
\end{equation*}
where $\chi_3$ is the function defined in \eqref{chi}. Therefore, $\sigma[L_1 X]$ satisfies the global definition of a symbol (as in \cite{Xavier}), rather than only the local condition on compact sets (cf. \cite{HormanderI}). Let $M$ and $N$ denote the multiplication operators associated with the functions $1 - \mu$ and $\nu$, respectively. Then, by the composition formula, the expansion of $\sigma[ML_1XN]$ contains the following terms 
\begin{equation*}
    \frac{1}{\alpha !} \frac{1}{\beta !} \partial _{(\tau,\xi_1,\xi_2)}^\alpha(1-\mu) D_{(t,x_1,x_2)}^\alpha \left(   \partial _{(\tau,\xi_1,\xi_2)}^\beta \sigma[L_1X](\tau,\xi_1,\xi_2)  D_{(t,x_1,x_2)}^\beta \nu (t,x_1,x_2) \right).
\end{equation*}
Therefore, since the supports of $1 - \mu$ and $\nu$ are disjoint, we conclude that $\sigma[ML_1XN] \in S^{-\infty}$. Applying these arguments to the other entries of $\mathcal{L}$, we derive that $M\mathcal{L}XN$ is a smoothing operator. We fix $k\in \mathbb{N}$ and $T'>0$. Since $v\in (I(K_\varepsilon))^4$, there is $s>0$ such that $u$ belongs to $ H^{-s}((0,T')\times \mathbb{R}^{2n};\mathbb{C}^4)$; see Proposition 3.6 in \cite{Xavier}. Since $\mathrm{supp}(v) \subset U_3$, we derive that
\begin{equation*}
    h_1 = M\mathcal{L}XNv \in H^k((0,T')\times \mathbb{R}^{2n}; \mathbb{C}^4),
    \qquad
    \text{for all } k\in \mathbb{N}.
\end{equation*}
Therefore,
\begin{equation*}
    h_1\in C^\infty((0,T')\times \mathbb{R}^{2n}) \cup L^2((0,T') ; H^k(\mathbb{R}^{2n}; \mathbb{C}^4)),
    \quad
    \partial_th_1 \in L^2((0,T') ; H^{k-1}(\mathbb{R}^{2n}; \mathbb{C}^4)).
\end{equation*}
By Theorem 3.1 in \cite{LionsMagenes}, we obtain 
\begin{equation*}
    h_1\in C((0,T'); H^{k-1}(\mathbb{R}^{2n}; \mathbb{C}^4) ).
\end{equation*}
Since $k\in\mathbb{N}$ and $T'>0$ were arbitrary, we deduce that $h_1 \in C^\infty(\mathbb{R}^{2n+1})$. Consequently, as $h_1+h_2 \in C^\infty(\mathbb{R}^{2n+1})$, it follows that $h_2 \in C^\infty(\mathbb{R}^{2n+1})$ as well. Moreover, for fixed $t>0$, $h_2(t,\cdot)$ is compactly supported. Therefore, we conclude that 
\begin{equation*}
    (F- \mathcal{A}_Xv) = h_1 + h_2 \in C((0,T'); H^k(\mathbb{R}^{2n}; \mathbb{C}^4)),
\end{equation*}
for any $k\in \mathbb{N}$ and $T'>0$. Then, Proposition \ref{cor: stability_for_aux_eq} implies that $w\in C^\infty(\mathbb{R}^{2n+1})$, and hence, $u = \chi v + w\in (I(K_\varepsilon))^4$ and $\sigma[u]$ is zero on $N^*K_\varepsilon \setminus N_\varepsilon^* K_\varepsilon$.
\end{proof}

\begin{remark}\label{Remark 1 on entanglement}
To solve the inverse problem, we will use sources of the form described in Lemma~\ref{par_main}.
Such sources cannot be factorizable. Indeed, for a factorizable source of the form $F(t,x_1,x_2)=F_1(t,x_1)F_2(t,x_2)$, the wavefront set satisfies
$\tau = |\xi_1| = |\xi_2|$.  Physically, this means that the source $F$ produces 
two entangled particles.
For the sources in Lemma~\ref{par_main}, we have $\operatorname{WF}(F)\subset N^*_\varepsilon K_\varepsilon$. On $N^*_\varepsilon K_\varepsilon$, one has $\tau = |\xi_1| + |\xi_2|$ with $|\xi_j| \ge \varepsilon > 0$.
Hence the wavefront set of any factorizable source cannot intersect $N^*_\varepsilon K_\varepsilon$.
\end{remark}

Next, we investigate the principal symbol of the solution to \eqref{auxiliary_eq}. Let $F$ be a conormal distribution satisfying the assumptions of Lemma \ref{par_main}, and suppose that $u = u^F$ solves equation \eqref{auxiliary_eq}. Then,
\begin{equation}\label{det_system}
    \begin{cases}
        D u_0 = g\sqrt{\rho_2} u_1 + g \sqrt{\rho_1} u_2 + F, \\
        (D_1 - \Omega)u_1 = g\sqrt{\rho_2} u_0 + g \sqrt{\rho_1} u_3, \\
        (D_2 - \Omega)u_2 = g\sqrt{\rho_1} u_0 - g \sqrt{\rho_2} u_3, \\
        (D_3 - 2\Omega)u_3 = g\sqrt{\rho_1} u_1 - g \sqrt{\rho_2} u_2.
    \end{cases}
\end{equation}
Since $\sigma[D_3]$ does not vanish on $N^*K_\varepsilon$, there exists a parametrix $E_3$ for the operator $D_3 - 2\Omega$, a pseudodifferential operator of order $-1$. In particular,
\begin{equation*}
    E_3(D_3 - 2\Omega) v = v + R_3 v,
    \qquad
    \text{for } v\in I(K_\varepsilon),
\end{equation*}
for some smoothing operator $R_3$. By Lemma \ref{par_main}, $u \in (I(K_\varepsilon))^4$. Hence, if we apply the operator $E_3$ to both sides of the last equation of the system \eqref{det_system}, we obtain
\begin{equation*}
    u_3 = E_3 \left(g\sqrt{\rho_1} u_1 - g\sqrt{\rho_2}u_2\right) - R_3 u_3.
\end{equation*}
We will consider the upcoming equations modulo smooth functions. Therefore, we denote by $R$ a smooth function that may vary from line to line. Since $D_3$ involves derivative with respect to $t$ variable only, we can rewrite the above identity as follows 
\begin{equation*}
    u_3 =  \left(g\sqrt{\rho_1} E_3 u_1 - g\sqrt{\rho_2} E_3 u_2\right) + R.
\end{equation*}
If we insert the last expression for $u_3$ into the second and third equations of the system \eqref{det_system}, we obtain 
\begin{align}\label{det_u_1}
    &(D_1 - \Omega - g^2\rho_1 E_3)u_1 = g\sqrt{\rho_2} u_0 - g^2\sqrt{\rho_1\rho_2}E_3u_2 + R,\\
    &\nonumber (D_2 - \Omega - g^2\rho_2 E_3)u_2 = g\sqrt{\rho_1} u_0 - g^2\sqrt{\rho_1\rho_2}E_3u_1 + R.
\end{align}
Since the principal symbol of $D_2$ does not vanish on $N^*K$, there exists a parametrix $E_2$ for the operator
\begin{equation*}
    D_2 - \Omega - g^2\rho_2 E_3.
\end{equation*}
Then, the last equation implies 
\begin{equation*}
    u_2 = g \sqrt{\rho_1} E_2 u_0 - g^2 E_2(\sqrt{\rho_1\rho_2} E_3 u_1) + R.
\end{equation*}
Here, we used the properties that $\rho_1$ depends only on the $x_1$ variable and $D_2$ is a differential operator with respect to $t$ and $x_2$ variables, so that the multiplication operator $\sqrt{\rho_1}$ and $E_2$ commute. We insert the above expression for $u_2$ into \eqref{det_u_1} to obtain 
\begin{equation*}
    \left(D_1 - \Omega - g^2\rho_1 E_3 - g^4 \sqrt{\rho_1\rho_2} E_3E_2(\sqrt{\rho_1\rho_2} E_3\cdot)  \right)u_1 = g\sqrt{\rho_2} u_0 - g^3\rho_1 \sqrt{\rho_2} E_3 E_2 u_0  + R.
\end{equation*}
Likewise, because the principal symbol of $D_1$ remains non-vanishing on $N^*K_\varepsilon$, we obtain a parametrix $E_1$ and hence
\begin{equation*}
    u_1 = g\sqrt{\rho_2} E_1 u_0 - g^3E_1(\rho_1 \sqrt{\rho_2} E_3 E_2 u_0)  + R.
\end{equation*}
Similarly, we derive
\begin{equation*}
    u_2 = g\sqrt{\rho_1} E_2' u_0 - g^3E_2'(\sqrt{\rho_1} \rho_2  E_3 E_1' u_0)  + R,
\end{equation*}
where $E_1'$ and $E_2'$ are parametrices for the following operators
\begin{equation*}
    D_1 - \Omega - g^2\rho_1 E_3,
    \qquad
    D_2 - \Omega - g^2\rho_2 E_3 - g^4 \sqrt{\rho_1\rho_2} E_3E_1'(\sqrt{\rho_1\rho_2} E_3\cdot),
\end{equation*}
respectively. Substituting the expressions for $u_1$ and $u_2$ into the first equation of the system \eqref{det_system}, we obtain
\begin{equation*}
Du_0 = Qu_0 + F + R,
\end{equation*}
where $Q$ is given by
\begin{equation}\label{det_Q}
Q = g^2(\rho_2E_1 + \rho_1E_2') - g^4 E_1(\rho_1 \rho_2 E_3 E_2 \cdot) - g^4 E_2'(\rho_1 \rho_2 E_3 E_1' \cdot).
\end{equation}
Assume that $u_0^{in}$ solves 
\begin{equation}\label{u_in}
\begin{cases}
    Du_0^{in} = F,\\
    u_0^{in}\arrowvert_{t<0}=0.
\end{cases}
\end{equation}
We define 
\begin{equation}\label{u_sc}
    u_0^{sc} = u_0 - u_0^{in}.
\end{equation}
Then,
\begin{equation*}
    (D - Q) u_0^{sc} = Qu_0^{in} + R.
\end{equation*}
Consider the bicharacteristic curve which goes through $(0, 0,0,\tau^*,\xi_1^*,\xi_2^*) \in \mathrm{Char}(D_{X})$:
\begin{align*}
    \beta(s) = \left(s,  x_1(s), x_2(s), \tau^*, \xi_1^*,\xi_2^* \right),
\end{align*}
where
\begin{equation*}
    x_1(s) = -\frac{\xi_1^*}{|\xi_1^*|}s,
    \qquad
    x_2(s) = -\frac{\xi_2^*}{|\xi_2^*|}s.
\end{equation*}
Then, by \cite[Th. 25.2.4]{hormander1985},
\begin{equation}\label{eq:LuQu}
    i^{-1} \mathcal{L}_H \sigma[u_0^{sc}](\beta(s)) = \sigma[Q](\beta(s)) \sigma[u_0^{in}](\beta(s)).
\end{equation}
Let $\omega$ be a half-density and suppose $\omega > 0$. Then, there is a smooth function $a$ such that $\sigma[u_0^{sc}] = a\omega$. Let $\alpha = a\circ \beta$. Then,
\begin{equation*}
    (\omega^{-1}\mathcal{L}_H (a\omega) )\circ \beta = e^{ - \vartheta} \partial_s (e^{\vartheta}\alpha)
\end{equation*}
where 
\begin{equation*}
    \vartheta(s) = \int_0^s \omega^{-1} (\beta(r)) \mathcal{L}_H \omega(\beta(r)) dr.
\end{equation*}
By combining this with \eqref{eq:LuQu}, we obtain 
\begin{equation*}
    \partial_s (e^{\vartheta}\alpha) = i e^\vartheta (\omega^{-1}\circ\beta)  (\sigma[Q]\circ\beta) (\sigma[u_0^{in}]\circ\beta).
\end{equation*}
By integrating this over $(0,t)$, we find that
\begin{equation*}
    \sigma[u_0^{sc}](\beta(t)) = i \omega(\beta(t)) e^{-\vartheta(t)} \int_{0}^t  \omega^{-1}(\beta(s)) e^{\vartheta(s)} \sigma[Q](\beta(s)) \sigma[u_0^{in}](\beta(s)) ds. 
\end{equation*}
Recalling the definition of $Q$, given by \eqref{det_Q}, we write
\begin{equation*}
    \sigma[Q] (\beta (t)) = g^2 \frac{\rho(x_1(s))}{\tau^* - |\xi_2^*|} + g^2 \frac{\rho(x_2(s))}{\tau^* - |\xi_1^*|}.
\end{equation*}
Therefore, 
\begin{multline}\label{sigma_u_sc}
    \sigma [u_0^{sc}] (\beta(t)) \\
    = i g^2\omega(\beta(t)) e^{-\vartheta(t)} \int_{0}^t \omega^{-1}(\beta(s)) e^{\vartheta(s)} \left( \frac{\rho(x_1(s))}{\tau^* - |\xi_2^*|} +  \frac{\rho(x_2(s))}{\tau^* - |\xi_1^*|} \right) \sigma[u_0^{in}](\beta(s)) ds.
\end{multline}
By repeating these steps except for $u_0^{in}$, we obtain
\begin{equation}\label{princ_symb_u_in}
    \sigma[u_0^{in}](\beta(t)) = i \omega(\beta(t)) e^{-\vartheta(t)} \int_{0}^t  \omega^{-1}(\beta(s)) e^{\vartheta(s)} \sigma[F](\beta(s))ds. 
\end{equation}

We introduce polar coordinates for $x_k$ as  
\begin{equation*}
  x_k = r_k \omega_k, \quad r_k > 0, \quad \omega_k \in \mathbb{S}^{n-1}.  
\end{equation*}
Next, we perform a change of variables given by  
\begin{equation}\label{change_variables}
    (t, x_1, x_2) \mapsto (t, r_1, \omega_1, r_2, \omega_2) \mapsto (t, s_1, \omega_1, s_2, \omega_2),
\end{equation}
where we define \( s_k = r_k - t \). Under this transformation, the phase space coordinates transform as  
\begin{equation*}
    (t, x_1, x_2, \tau, \xi_1, \xi_2) \mapsto (t, s_1, \omega_1, s_2, \omega_2, \eta, \sigma_1, \xi_{\omega_1}, \sigma_2, \xi_{\omega_2}),
\end{equation*}
with the new variables given by  
\begin{equation*}
    \omega_k = \frac{x_k}{|x_k|}, \qquad
\sigma_k = \xi_k \cdot \omega_k, \qquad
\xi_{\omega_k} = \xi_k - (\xi_k\cdot \omega_k) \omega_k,
\end{equation*}
\begin{equation*}
    \eta = \tau + \xi_1\cdot \omega_1 + \xi_2\cdot \omega_2.
\end{equation*}
In these coordinates, the bicharacteristic curve passing through  $(0, 0, 0, \tau, \xi_1, \xi_2)$ takes the form  
\begin{equation*}
\beta(t) = \left( t, 0, \omega_1, 0, \omega_2, 0, \sigma_1, 0, \sigma_2, 0 \right),
\end{equation*}
where $\sigma_k = - |\xi_k|$ and $\omega_k = - \xi_k/|\xi_k|$. Finally, in these coordinates, we write  
\begin{align*}
N_\varepsilon^*K_\varepsilon = \biggl\{ (t, 0, \omega_1, 0, \omega_2, 0, -\sigma_1, 0, -\sigma_2, 0) \mid 
t > \varepsilon, \ (\omega_1, \omega_2) \in \mathbb{S}^{n-1} \times \mathbb{S}^{n-1}, \ \sigma_1, \sigma_2 > \varepsilon \biggr\}.
\end{align*}

Rewriting the expressions \eqref{sigma_u_sc} and \eqref{princ_symb_u_in} in the new variables, we obtain the following result.
\begin{lemma}\label{l: Pr_symb_u_in_sc}
    Let $u_0^{in}$ and $u_0^{sc}$ be the functions given by \eqref{u_in} and \eqref{u_sc}, respectively. Let $\beta$ be the bicharacteristic curve passing through $\left( t, 0, \omega_1, 0, \omega_2, 0, \sigma_1, 0, \sigma_2, 0 \right)$. Then,
    \begin{multline*}
        \sigma [u_0^{sc}] (t, 0, \omega_1, 0, \omega_2, 0, \sigma_1, 0, \sigma_2, 0) \\
        =ig^2 V(t,\omega_1,\omega_2) \int_0^t \frac{1}{V(s, \omega_1,\omega_2)} \left( \frac{\rho(s\omega_1)}{-\sigma_1} +  \frac{\rho(s\omega_2)}{-\sigma_2} \right) \sigma[u_0^{in}](\beta(s))ds
    \end{multline*}
    and
    \begin{equation*}
        \sigma [u_0^{in}] (t, 0, \omega_1, 0, \omega_2, 0, \sigma_1, 0, \sigma_2, 0) 
        = iV(t,\omega_1,\omega_2) \int_0^t \frac{1}{V(s, \omega_1,\omega_2)}  \sigma[F](\beta(s)) ds,
    \end{equation*}
    where 
    \begin{equation}\label{A}
        V(t,\omega_1,\omega_2) =  \omega(\beta(t)) e^{-\vartheta(t)}. 
    \end{equation}
\end{lemma}

Next, we introduce some notation. For functions $u$ and $v$ for which the following integrals are well-defined, we set
\begin{align}\label{def_form}
    m(u, v)(t, x_1) &= \int_{\mathbb{R}^n} \chi(x_2)\, u(t, x_1, x_2)\, \overline{v(t, x_1, x_2)}\, dx_2, \\
    \nonumber m(u)(t, x_1) &= m(u, u)(t, x_1).
\end{align}
Furthermore, we define the submanifold $K_\varepsilon^1 \subset \mathbb{R}^{n+1}$ in the coordinates given by~\eqref{change_variables} as
\begin{equation*}
    K_\varepsilon^1 = \{ (t, s_1, \omega_1) :\; t > \varepsilon \text{ and } s_1 = 0 \}.
\end{equation*}

\begin{lemma}\label{distr_R3}
    Let $u\in I(K_\varepsilon)$ and $v\in C^\infty(\mathbb{R}^{2n+1})$, then $m(u,v) \in I(K_\varepsilon^1)$. Moreover, in the coordinates given by \eqref{change_variables}, 
    \begin{multline*}
        \sigma[m(u,v)](t,\omega_1;\sigma_1)\\
        = \int_{\mathbb{R}} \int_{\mathbb{S}^{n-1}} e^{-i\sigma_2 t} [\mathcal{F}^{-1} \eta] (\sigma_2, \omega_2)  \overline{v(t, 0, \omega_1, 0, \omega_2)} \sigma[u](t, \omega_1, \omega_2; \sigma_1, \sigma_2)  d\omega_2d\sigma_2,
    \end{multline*}
    where the function $\eta$ is defined as $\eta(s,\omega) = s\chi(s,\omega)$ and $\mathcal{F} = \mathcal{F}_{\sigma_2\mapsto s_2}$ denotes the Fourier transform. 
\end{lemma}

\begin{proof}
Since $v \in C^\infty(\mathbb{R}^{2n+1})$, it follows from the definition of conormal distributions that $u\overline{v} \in I(K_\varepsilon)$ and that $\sigma[u\overline{v}] = \overline{v}\sigma[u]$, and hence, 
\begin{multline*}
        \overline{v(t,0,\omega_1, 0,\omega_2)}u(t, s_1, \omega_1, s_2, \omega_2) \\ 
        = \int_{\mathbb{R}^2} e^{i (\sigma_1 s_1 + \sigma_2 s_2)} \overline{v(t,0,\omega_1, 0,\omega_2)}\sigma[u](t, \omega_1, \omega_2; \sigma_1, \sigma_2) d\sigma_1 d\sigma_2.
\end{multline*}
Therefore, in the coordinates \eqref{change_variables}, we express $m(u,v)$ as  
\begin{align*}
    m(u,v)(t,s_1,\omega_1) & = \int_{\mathbb{R}} \int_{\mathbb{S}^{n-1}} (s_2 + t)\chi(s_2 + t, \omega_2) \overline{v(t,0,\omega_1, 0,\omega_2)}u(t,s_1,\omega_1, s_2,\omega_2) ds_2 d\omega_2\\
    & = \int_{\mathbb{R}} e^{i\sigma_1s_1} b(t,\omega_1, \sigma_1) d\sigma_1,
\end{align*}
where 
\begin{equation*}
    b(t,\omega_1, \sigma_1) = \int_{\mathbb{R}} \int_{\mathbb{S}^{n-1}} e^{-i\sigma_2 t} [\mathcal{F}^{-1} \eta] (\sigma_2, \omega_2)  \overline{v(t,0,\omega_1, 0,\omega_2)} \sigma[u](t, \omega_1, \omega_2; \sigma_1, \sigma_2)  d\omega_2d\sigma_2.
\end{equation*}
To show that $m(u,v)$ is a conormal distribution, it remains to show that $b$ is a symbol. 

Let $\alpha$, $\beta$ be multi-indexes, then $D_{t,\omega_1}^\alpha D_{\sigma_1}^\beta b$ is a linear combination of the following terms 
\begin{equation*}
    b_{l,\tau,\beta} (t,\omega_1,\sigma_1) = \int_{\mathbb{R}} \int_{\mathbb{S}^{n-1}} \sigma_2^l e^{-i\sigma_2 t} [\mathcal{F}^{-1} \eta] (\sigma_2, \omega_2)  D_{t,\omega_1}^\tau D_{\sigma_1}^\beta \sigma[u](t, \omega_1, \omega_2; \sigma_1, \sigma_2)  d\omega_2d\sigma_2,
\end{equation*}
where $0\leq l\leq |\alpha|$ and $\tau\leq \alpha$. By the definition of a symbol, for any compact set $K\subset \mathbb{R}\times \mathbb{S}^{n-1}$, there is a constant $C_{\alpha,\beta, K}> 0$ such that 
\begin{equation*}
    D_{t,\omega_1}^\tau D_{\sigma_1}^\beta \sigma[u](t, \omega_1, \omega_2; \sigma_1, \sigma_2) \leq C_{\alpha,\beta, K} (1 + |(\sigma_1,\sigma_2)|)^{m - |\beta|},
\end{equation*}
for $(t,\omega_1) \in K$, $\omega_2\in \mathbb{S}^{n-1}$, and $\sigma_1$, $\sigma_2\in \mathbb{R}$, where $m$ is the order of $\sigma[u]$. Peetre's inequality, Lemma 1.18 in \cite{Xavier}, implies
\begin{equation*}
    (1 + |(\sigma_1,\sigma_2)|)^{m - |\beta|} \leq C_{m,\beta} (1 + |\sigma_1|)^{m - |\beta|} (1 + |\sigma_2|)^{|m - |\beta||}.
\end{equation*}
Therefore, 
\begin{multline*}
    |b_{l,\tau,\beta} (t,\omega_1,\sigma_1)|\\
    \leq C_{\alpha,\beta,K} (1 + |\sigma_1|)^{m - |\beta|} \int_{\mathbb{R}} \int_{\mathbb{S}^{n-1}} |\sigma_2|^l |[\mathcal{F}^{-1} \eta] (\sigma_2, \omega_2)|  (1 + |\sigma_2|)^{|m - |\beta||}  d\omega_2d\sigma_2.
\end{multline*}
Since $\chi$ is compactly supported, $\eta$ decays in $\sigma_2$ faster than any polynomial. Hence,
\begin{equation*}
    |b_{l,\tau,\beta} (t,\omega_1,\sigma_1)| \leq C_{\alpha,\beta,Q} (1 + |\sigma_1|)^{m - |\beta|}.
\end{equation*}
This implies that $b$ is a symbol and $m(u,v) \in I(K_\varepsilon^1)$.
\end{proof}

\section{Inverse problem}
In this section, we resolve the inverse problem. We begin with some preliminary results. 

\begin{lemma}\label{polarization}
    Let $f$, $h\in\mathcal{C}_{\operatorname{sym}}(S)$ and $u^f$, $u^h$ be the corresponding solutions of \eqref{main_eq}, respectively. Then, the measurement map $\Lambda$ determines $m(u_0^f,u_0^h)$ on $(0,\infty)\times W_1$, where $m(\cdot,\cdot)$ is defined in \eqref{def_form}.
\end{lemma}
\begin{proof}
    Observe that 
    \begin{equation*}
        m(u_0^f, u_0^h) = \frac{1}{4} \left( 
        m(u_0^f + u_0^h) - m(u_0^f - u_0^h) +
        i\, m(u_0^f + iu_0^h) - i\, m(u_0^f - iu_0^h)
        \right).
    \end{equation*}
    For any $f\in\mathcal{C}_{\operatorname{sym}}(S)$, we know that $m(u_0^f) = \Lambda f$. Hence, since \eqref{main_eq} is linear, the map $\Lambda$ determines each term on the right-hand side, and hence, determines $ m(u_0^f, u_0^h)$.
\end{proof}

Let $k \in \mathbb{N}$ be sufficiently large and let $\varepsilon > 0$ be sufficiently small. Assume that $F \in H^k(\mathbb{R} \times \mathbb{R}^{2n})$ satisfies conditions of Lemma \ref{par_main} with $T = 2\varepsilon$. Then, by Proposition \ref{cor: stability_for_aux_eq} and Lemma \ref{par_main}, there exists a unique solution $u=u^F$ such that $u \in \left( I(K_\varepsilon) \right)^4$ and $u \in C(\mathbb{R}; H^k(\mathbb{R}^{2n}; \mathbb{C}^4))$.

Let $\mathcal{A}_0$ be the operator $\mathcal{A}$ but with $\rho = 0$. Proposition \ref{cor: stability_for_aux_eq} implies that there exists unique solution $w \in C(\mathbb{R}; H^k(\mathbb{R}^{2n}; \mathbb{C}^4))$ of the equation
\begin{equation}\label{eq:for_omega}
\begin{cases}
\mathcal{A}_0 w = F,\\
w\arrowvert_{t <0} = 0.
\end{cases}
\end{equation}
Moreover, by Lemma~\ref{par_main}, $w \in \left( I(K_\varepsilon) \right)^4$. Let $\varkappa \in C_0^\infty(\mathbb{R}^{2n})$ be a function such that $\varkappa = 1$ in $B_{4\varepsilon}(0)$, and let $u'$ be the solution of
\begin{equation}\label{shifted_Av_eq_0}
\begin{cases}
\mathcal{A} u' = 0 & \text{for } t > 3\varepsilon,\\
u'\arrowvert_{t = 3\varepsilon} = \varkappa w\arrowvert_{t = 3\varepsilon}.
\end{cases}
\end{equation}

Next, we show that $u - u'$ is smooth for $t>3\varepsilon$.

\begin{lemma}\label{l:dif_smooth}
Assume that $\Sigma$ does not contain the origin and $\varepsilon>0$ is sufficiently small. Let $F$, $u$, and $u'$ be as above. Then, $u - u' \in C^\infty((3\varepsilon, \infty) \times \mathbb{R}^{2n})$.
\end{lemma}

\begin{proof}
Since $\Sigma$ does not contain the origin, for sufficiently small $\varepsilon > 0$, it follows
\begin{equation*}
\left(K_\varepsilon \cap \{t < 4\varepsilon\}\right)  \bigcap (0, 4\varepsilon) \times 
\left( (\Sigma \times \mathbb{R}^n) \cup (\mathbb{R}^n \times \Sigma) \right) = \varnothing,
\end{equation*}
and consequently,
\begin{equation}\label{eq:sing_sup_w}
\left( \operatorname{sing\,supp}(w) \cap \{t < 4\varepsilon\} \right)
\bigcap (0, 4\varepsilon) \times 
\left( \operatorname{supp}(\rho_1) \cup \operatorname{supp}(\rho_2) \right) = \varnothing.
\end{equation}
We observe that
\begin{equation*}
\mathcal{A} - \mathcal{A}_0 =
\begin{bmatrix}
0 & g \sqrt{\rho_2} & g \sqrt{\rho_1} & 0 \\
g \sqrt{\rho_2} & 0 & 0 & g \sqrt{\rho_1} \\
g \sqrt{\rho_1} & 0 & 0 & -g \sqrt{\rho_2} \\
0 & g \sqrt{\rho_1} & -g \sqrt{\rho_2} & 0
\end{bmatrix}.
\end{equation*}
Therefore, \eqref{eq:sing_sup_w} implies that $(\mathcal{A} - \mathcal{A}_0)w$ is smooth in $(0, 4\varepsilon) \times \mathbb{R}^{2n}$.

Next, let $w'$ be the solution of the equation
\begin{equation*}
\begin{cases}
\mathcal{A} w' = 0 & \text{for } t > 3\varepsilon,\\
w' \arrowvert_{t = 3\varepsilon} = (1 - \varkappa) w \arrowvert_{t = 3\varepsilon}.
\end{cases}
\end{equation*}
By our choice of $\varkappa$ and $w$, we know that $(1 - \varkappa) w\arrowvert_{t=3\varepsilon}$ is smooth, and hence, $w'$ is also smooth.
We note
\begin{equation*}
    \begin{cases}
        \mathcal{A}(w - u' - w') = (\mathcal{A} - \mathcal{A}_0) w, & \text{for } t>3\varepsilon,\\
        (w - u' - w') \arrowvert_{t = 3\varepsilon} = 0
\end{cases}
\qquad
    \begin{cases}
        \mathcal{A}(w - u) = (\mathcal{A} - \mathcal{A}_0) w,\\
        (w - u) \arrowvert_{t < 0} = 0
\end{cases}
\end{equation*}
As we concluded previously, $(\mathcal{A} - \mathcal{A}_0)w$ is smooth in $(0, 4\varepsilon) \times \mathbb{R}^{2n}$, and hence, 
$w - u' - w'$ and $w - u$ are smooth there as well. Therefore, the difference
\begin{equation*}
u' - u = (u' + w' - w) + (w - u) - w'
\end{equation*}
is smooth in $(3\varepsilon, 4\varepsilon) \times \mathbb{R}^{2n}$. Noting that $\mathcal{A}(u' - u) = 0$, we conclude that $u' - u$ is smooth in $(3\varepsilon, \infty) \times \mathbb{R}^{2n}$.
\end{proof}

\begin{lemma}\label{l: det_aux_funct}
    Let $h \in \mathcal{C}_{\operatorname{sym}}(S)$, and let $v = v^h$ be the solution of \eqref{main_eq} with the initial data $h$. Assume that $S \setminus \Sigma$ contains the origin in $\mathbb{R}^{2n}$. Then $\Lambda$ determines
    \begin{equation*}
        v_0(t, t\omega_1, t\omega_2)
    \end{equation*}
    for all $t > 0$ and $\omega_1, \omega_2 \in \mathbb{S}^{n-1}$ such that $\omega_1 \neq \omega_2$, $t\omega_1 \in W_1$, and $s\omega_2 \in W_2$ for some $s > 0$.
\end{lemma}

\begin{proof}
    Let us fix $t'>\varepsilon$ and $\omega_1^*$, $\omega_2^* \in \mathbb{S}^{n-1}$ such that $t'\omega_1^* \in  W_1$ and $s\omega_2^*\in W_2$ for some $s>0$. Then, under the change of variables, given by \eqref{change_variables}, we get
    \begin{align*}
        (t', t'\omega_1^*, t'\omega_2^*) \mapsto (t', 0, \omega_1^*, 0, \omega_2^*).
    \end{align*}
    Hence, it remains to determine $v_0(t', 0, \omega_1^*, 0, \omega_2^*)$ from the data. We set $t^* = t' + 3\varepsilon$. Let $v'$ be the solution of the shifted initial-value problem
    \begin{equation*}
        \begin{cases}
            \mathcal{A}v' = 0 & \text{for } t>3\varepsilon,\\
            v'\arrowvert_{t=3\varepsilon} = h.
        \end{cases}
    \end{equation*}
    Then $v'(t, \cdot) = v(t - 3\varepsilon,\cdot)$, and hence, it sufficient to determine $v_0'(t^*, 0, \omega_1^*, 0, \omega_2^*)$ from the data.

    Next, we construct a source. Let $\mu \in C_0^\infty(\mathbb{R})$ be a function supported in $(\varepsilon, 2\varepsilon)$ such that 
\begin{equation*}
    V(t^*,\omega_1^*,\omega_2^*) \int_{\varepsilon}^{t^*}  \frac{\mu(s)}{V(s,\omega_1^*,\omega_2^*)}ds = 1,
\end{equation*}
where $V$ is the function defined by \eqref{A}. Let $\Theta$ be a smooth function on $\mathbb{S}^{n-1}$, to be specified later, and $\nu$ be a smooth function on $\mathbb{R}$ such that 
\begin{equation}\label{nu}
    \nu(\sigma) = \nu_\varepsilon (\sigma) = 
    \begin{cases}
        0 & \text{if } \sigma<\varepsilon,\\
        1/\sigma^l & \text{if } \sigma > 1.
    \end{cases}
\end{equation}
The function $\nu$ will be used in \eqref{source} to define the symbol of a source; the decay factor $1/\sigma^l$ for large frequencies ensures that the symbol belongs to the required class, while the vanishing of $\nu$ for $\sigma < \varepsilon$ guarantees that this source satisfies the assumptions of Lemma \ref{par_main}.

We fix $l \in \mathbb{N}$ sufficiently large. Let $F_0 \in I(K_\varepsilon)$ be a function supported in $(\varepsilon,2\varepsilon)\times S$ and symmetric with respect to the variables $x_1$ and $x_2$, with principal symbol given by 
\begin{equation}\label{source}
    \sigma[F_0](t,\omega_1,\omega_2, \sigma_1, \sigma_2) = \mu(t) (\Theta(\omega_1) + \Theta(\omega_2) )\nu(\sigma_1) \nu(\sigma_2).  
\end{equation}
Let $F= [F_0,0,0,0]^T$ and $u = u^F$ be the corresponding solution of \eqref{auxiliary_eq}. From Lemmas \ref{par_main} and \ref{distr_R3}, it follows that $m(u_0, v_0')\in I(K^1_\varepsilon)$ and for $t>3\varepsilon$,
\begin{multline*}
    \sigma[m(u_0,v_0')](t,\omega_1,\sigma_1) = \sigma[m(u_0^{in},v_0')](t,\omega_1,\sigma_1)\\
    = \int_{\mathbb{R}} \int_{\mathbb{S}^{n-1}} e^{-i\sigma_2 t} [\mathcal{F}^{-1} \eta] (\sigma_2, \omega_2)  \overline{v_0'(t, 0,\omega_1,0,\omega_2)}\sigma[u_0^{in}](t, \omega_1, \omega_2; \sigma_1, \sigma_2)  d\omega_2d\sigma_2,
\end{multline*}
where $u_0^{in}$ is the function defined by \eqref{u_in}.  Due to  Lemma \ref{l: Pr_symb_u_in_sc}, we know that 
\begin{multline}\label{eq: princ_symb_u_in}
    \sigma[u_0^{in}] (t, \omega_1, \omega_2, \sigma_1, \sigma_2)\\
    = i V(t, \omega_1,\omega_2) \int_0^{t} \frac{\mu(s)}{V(s, \omega_1,\omega_2)}  (\Theta(\omega_1) + \Theta(\omega_2) ) \nu(\sigma_1) \nu(\sigma_2) ds,
\end{multline}
for $t>3\varepsilon$. Therefore, 
\begin{multline*}
    \sigma[m(u_0,v_0')](t,\omega_1,\sigma_1) = \int_{\mathbb{R}} \int_{\mathbb{S}^{n-1}} e^{-i\sigma_2 t} [\mathcal{F}^{-1} \eta] (\sigma_2, \omega_2)  \overline{v_0'(t, 0,\omega_1,0,\omega_2)}\\
    \times i V(t, \omega_1,\omega_2) \int_0^{t} \frac{\mu(s)}{V(s, \omega_1,\omega_2)} (\Theta(\omega_1) + \Theta(\omega_2) ) \nu(\sigma_1) \nu(\sigma_2) ds  d\omega_2d\sigma_2,
\end{multline*}
for $t>3\varepsilon$. By letting $\Theta = \Theta_r$ approach $\delta_{\omega_2^*}$, the Dirac delta function centered at $\omega_2^*$, as $r \to 0$, we obtain
\begin{multline*}
    \lim_{r\rightarrow 0}\sigma[m(u_0,v_0')](t,\omega_1^*,\sigma_1) = \int_{\mathbb{R}} e^{-i\sigma_2 t} [\mathcal{F}^{-1} \eta] (\sigma_2, \omega_2^*)  \overline{v_0'(t, 0,\omega_1^*,0,\omega_2^*)}\\
    \times i V(t, \omega_1^*,\omega_2^*) \int_0^{t} \frac{\mu(s)}{V(s, \omega_1^*,\omega_2^*)} \nu(\sigma_1) \nu(\sigma_2) ds d\sigma_2.
\end{multline*}
Here, we used that $\omega_1^*\neq \omega_2^*$, so that $\Theta_r(\omega_1^*) =0$ as $r \to 0$. Therefore, due to the choice of $\mu$, we obtain
\begin{align*}
    \lim_{r\rightarrow 0}\sigma[m(u_0,v_0')](t^*,\omega_1^*,\sigma_1) &= \int_{\mathbb{R}} e^{-i\sigma_2 t^*} [\mathcal{F}^{-1} \eta] (\sigma_2, \omega_2^*)  \overline{v_0'(t^*, 0,\omega_1^*,0,\omega_2^*)} \nu(\sigma_1) \nu(\sigma_2)  d\sigma_2\\
    & = i \overline{v_0'(t^*, 0,\omega_1^*,0,\omega_2^*)} \nu(\sigma_1) \int_{\mathbb{R}} e^{-i\sigma_2 t^*} [\mathcal{F}^{-1} \eta] (\sigma_2, \omega_2^*) \nu(\sigma_2)  d\sigma_2.
\end{align*}
Let $u'$ be the solution of \eqref{shifted_Av_eq_0}. Then, by Lemma \ref{l:dif_smooth}, we know that $u - u'$ is smooth on $(3\varepsilon,\infty)\times \mathbb{R}^{2n}$. Therefore, 
\begin{multline}\label{eq: recov_v}
    \lim_{r\rightarrow 0}\sigma[m(u_0',v_0')](t^*,\omega_1^*,\sigma_1)\\
    = i \overline{v_0'(t^*, 0,\omega_1^*,0,\omega_2^*)} \nu(\sigma_1) \int_{\mathbb{R}} e^{-i\sigma_2 t^*} [\mathcal{F}^{-1} \eta] (\sigma_2, \omega_2^*) \nu(\sigma_2)  d\sigma_2.
\end{multline}
Since $u'$ and $v'$ are solutions of shifted initial-value problems, by Lemma \ref{polarization}, the data provide the left-hand side of the above identity. Moreover, the functions $\nu$ and $\eta$ are known. Therefore, we can recover $v_0'(t^*, 0,\omega_1^*,0,\omega_2^*)$ unless the integral on the right-hand side is zero. Consider the case where this integral vanishes, that is,
\begin{equation*}
    \mathcal{F}[[\mathcal{F}^{-1} \eta] (\cdot, \omega_2^*) \nu(\cdot)] (t^*) = 0.
\end{equation*}
As $s\omega_2^* \in W_2$ for some $s>0$, $\eta(\cdot,\omega_2^*)$ is a compactly supported, smooth, and nonzero function. Hence, we know that $\mathcal{F}^{-1} \eta(\cdot,\omega_2^*)$ is a Schwartz function and entire. Therefore, there exists an interval $(a,b) \subset (\varepsilon, \infty)$ where $\mathcal{F}^{-1} \eta(\cdot,\omega_2^*)$ does not vanish. Thus, we can find a smooth function $\phi$ supported in $(a,b)$ such that
\begin{equation*}
    \mathcal{F}[\mathcal{F}^{-1} \eta] (\cdot, \omega_2^*) \phi(\cdot)] (t^*) \neq 0.
\end{equation*}
Hence, if we repeat the arguments above for $\nu' = \nu + \phi$ instead of $\nu$, we obtain \eqref{eq: recov_v} with a non-zero integral on the right-hand side. This completes the proof.
\end{proof}

\begin{remark}\label{remak_shift}
    We observe that identity \eqref{eq: recov_v} holds for all smooth functions $v'$. Indeed, its derivation relies only on the smoothness of $v'$, while the assumption that it is a solution is used only to ensure that the left-hand side of \eqref{eq: recov_v} can be determined from the given data.
\end{remark}

To solve the inverse problem, we impose certain assumptions on the sets $W_1$, $W_2$, $\Sigma$, and $S$. To motivate these assumptions, we briefly outline the underlying strategy.
Consider $F$ and $h$, to be specified later, and let $u= u^F$ and $v = v^h$ be the corresponding solutions to \eqref{main_eq}. We decompose the pairing as
\begin{equation*}
    m(u_0^{sc}, v_0) = m(u_0, v_0) - m(u_0^{in}, v_0),
\end{equation*}
where $u_0^{in}$ and $u_0^{sc}$ are defined by \eqref{u_in} and \eqref{u_sc}, respectively. We aim to compute the full symbol of the left hand side. Due to Lemmas \ref{l:dif_smooth} and \ref{polarization}, the full symbol of $m(u_0, v_0)$ is determined by the measurement data. To compute the second term, we observe that $u_0^{in}$ is known, as it does not depend on the unknown density $\rho$. Therefore, it remains to determine $v_0$, which can be accomplished via Lemma~\ref{l: det_aux_funct}. To apply Lemma~\ref{l: det_aux_funct}, we require the following synchronization condition: for each point of interest $(t, x_1) \in (0, \infty) \times W_1$ and for every $x_2 \in W_2$, there exists a point $y = (y_1, y_2) \in S$ such that $|y_k - x_k| = t$ for $k = 1,2$. This ensures that the term $m(u_0^{in}, v_0)$ is determined from the data, and thus so is the left-hand side $m(u_0^{sc}, v_0)$.

Next, by applying Lemmas~\ref{l: Pr_symb_u_in_sc} and~\ref{distr_R3}, we show that for a suitable choice of $F$ and $h$, the term $m(u_0^{sc}, v_0)$ yields the quantity
\begin{equation*}
    \int_0^t \left( \rho(s\omega_1) + \rho(s\omega_2) \right) \, ds,
\end{equation*}
provided that $t\omega_1 \in W_1$ and $t\omega_2 \in W_2$.

To recover the partial X-ray transform data, we would like the ray $\{r\omega_1\}$ to avoid the set $\Sigma$, so that $\rho(s\omega_1) = 0$. On the other hand, we require that the rays $\{r\omega_2\}$ cover $\Sigma$, so that the integral provides information about $\rho$ on its support. These geometric requirements are encoded in a condition on the relative positions of $S$, $W_1$, $W_2$, and $\Sigma$ that we now describe.
 
To state the precise geometric assumptions, we introduce some notation. For $x, y \in \mathbb{R}^n$ and a set $W \subset \mathbb{R}^n$, we denote by $[x, y]$ the closed line segment connecting $x$ and $y$, and define
\begin{equation*}
	(x; W] = \bigcup_{p \in W} [x, p] \setminus \{x\}.
\end{equation*}

We now summarize the geometric requirements in the following condition.

\begin{condition}\label{condition_intro}
    There exists an open non-empty set $X_1 \subset W_1$, $T>0$, and $\ell\in (0,T)$ such that for every $x_1 \in X_1$, $x_2 \in W_2$, and $t\in (T-\ell, T+\ell)$, there exists $y = (y_1, y_2) \in S$ satisfying
\begin{equation*}
    |x_1 - y_1| = |x_2 - y_2| = t.
\end{equation*}
Additionally, there exists a point $z = (z_1, z_2) \in S$ such that:
\begin{enumerate}
    \item \label{1_condition} $\Sigma \subset (z_2; W_2]$.
    \item For every $x_2 \in W_2$, there exists a point $x_1 \in X_1$ such that:
    \begin{itemize}
        \item[a)] The segment $[z_1; x_1]$ does not intersect $\Sigma$, that is, $[z_1; x_1] \cap \Sigma = \varnothing$.
        \item[b)] The distances satisfy $|x_1 - z_1| = |x_2 - z_2|\in (T-\ell, T+\ell)$.
    \end{itemize}
\end{enumerate}
\end{condition}

\begin{figure}
	\centering
	\includegraphics[width=0.90\textwidth]{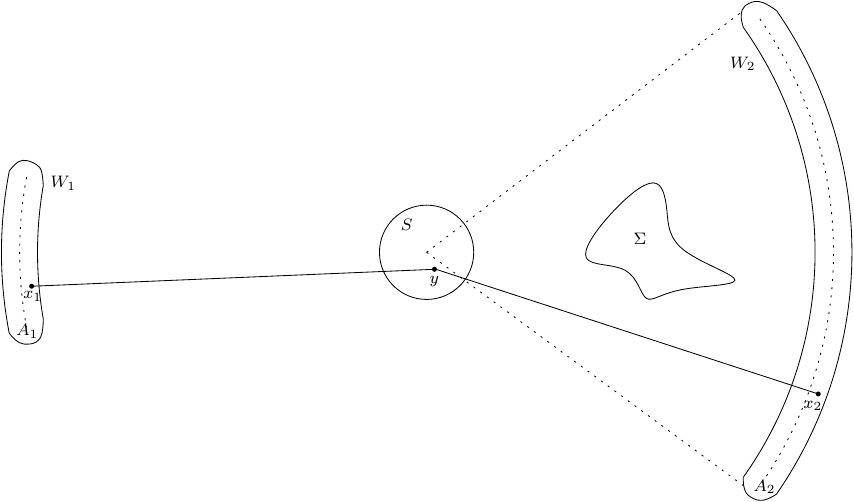}
	\caption{Illustrating the geometric configuration of the sets $W_1$, $W_2$, $S$, and $\Sigma$, with rays from the source point $y$ to detection points $x_1 \in W_1$ and $x_2 \in W_2$.}
	\label{fig:geometry}
\end{figure}

We next present a model configuration that illustrates
Condition~\ref{condition_intro}. Let $S \subset \mathbb{R}^{2n}$ be an open set containing the origin. 
Let $A_1$ and $A_2$ be two disjoint spherical caps on the sphere in $\mathbb{R}^n$ of radius $T > 0$. 
We assume that $\Sigma$ lies inside the open cone of height $T$ with vertex at the origin and spherical base $A_2$. 
Furthermore, let $W_1$ and $W_2$ denote the $\ell$-neighborhoods of $A_1$ and $A_2$, respectively. 
Assume that $\ell > 0$ is sufficiently small relative to $S$ so that for any $(x_1,x_2) \in W_1\times W_2$, one can find a point $y = (y_1, y_2) \in S$ such that $x_1$ and $x_2$ lie at the same distance from $y_1$ and $y_2$, respectively. This geometric configuration is illustrated in Figure \ref{fig:geometry}.

With these geometric assumptions in place, we establish several auxiliary results.

\begin{corollary}\label{cor:recover_v}
    Assume that Condition~\ref{condition_intro} holds, with the corresponding
set $X_1$ and parameters $T$ and $\ell$, and let $z$ be the point given there.
 Let $h \in \mathcal{C}_{\operatorname{sym}}(S)$, and let $v = v^h$ be the solution to \eqref{main_eq}. Then, for distinct points $x_1^* \in X_1$, $x_2^* \in W_2$, and any time $t^* \in (T - \ell, T + \ell)$, the given data determine $v_0(t^*, x_1^*, x_2^*)$.
\end{corollary}

\begin{proof}
    Let $x_1^* \in X_1$, $x_2^* \in W_2$, and $t^* \in (T - \ell, T + \ell)$. By Condition~\ref{condition_intro}, there exists a point $y = (y_1, y_2) \in S$ such that
    \begin{equation*}
        |x_1^* - y_1| = |x_2^* - y_2| = t^*.
    \end{equation*}
    Without loss of generality, we may assume that $y$ is the origin in $\mathbb{R}^{2n}$; otherwise, we can translate the construction of the cone $K_\varepsilon$ so that it originates from $y$ and proceed as before. Under this assumption, we have $|x_1^*| = |x_2^*| = t^*$, and the claim follows directly from Lemma~\ref{l: det_aux_funct}.
\end{proof}

We can now state the following lemma.

\begin{lemma}\label{l:cond}
Assume that Condition~\ref{condition_intro} holds, with the corresponding set $X_1$ and parameters $T$ and $\ell$, and let $z$ be the point given there. Then there exist a ball $Y_2$ in $\mathbb{R}^{n}$ centered at $z_2$ and an open subset $X_2 \subset W_2$ such that:
\begin{enumerate}
    \item \label{item:YinS} $Y = \{z_1\} \times Y_2 \subset S$.
    \item \label{item:SigmaCovered} For every $p \in Y_2$, we have $\Sigma \subset (p; X_2]$.
    \item \label{item:existenceOfx1} For every $y \in Y$ and $x_2 \in X_2$, there exists a point $x_1 \in X_1$ such that:
    \begin{itemize}
        \item[a)] \label{item:noIntersection} The segment $[y_1; x_1]$ does not intersect $\Sigma$, that is, $[y_1; x_1] \cap \Sigma = \varnothing$.
        \item[b)] \label{item:distanceEquality} The distances satisfy $|x_1 - y_1| = |x_2 - y_2| \in (T-\ell, T+\ell)$.
    \end{itemize}
    \item\label{item: Y2doesnotintersect} $Y_2\cap (X_2 \cup \Sigma)=\varnothing$.
\end{enumerate}
\end{lemma}

\begin{proof}
We define
\begin{equation*}
X_{2,\varepsilon} = \left\{ p \in W_2 : \; \operatorname{dist}(p, \partial W_2) > \varepsilon \right\}.
\end{equation*}
By Condition~\ref{condition_intro},
\begin{equation*}
\Sigma \subset (z_2; W_2] \subset \bigcup_{\varepsilon > 0}  (z_2; X_{2,\varepsilon}].
\end{equation*}
Since $\Sigma$ is compact and the right-hand side is a union of open sets, there exists $\varepsilon_1 > 0$ such that
\begin{equation*}
\Sigma \subset (z_2, X_{2,\varepsilon_1}].
\end{equation*}
We set $X_2' = X_{2,\varepsilon_1}$.

Next, we will show that for sufficiently small $\varepsilon_2>0$, 
\begin{equation}\label{eq_l:cond}
\Sigma \subset (q, X_2'] \quad \text{for all } q \in B_{\varepsilon_2}(z_2).
\end{equation}
Suppose, for the sake of contradiction, that there exist sequences $\{q_n\} \subset \mathbb{R}^n$ and $\{p_n\} \subset \Sigma$ such that $p_n \notin (q_n; X_2']$ and $q_n \rightarrow z_2$ as $n\rightarrow \infty$. Since $\Sigma$ is compact, we may assume (after passing to a subsequence) that $p_n \to p \in \Sigma$. Consider the open cone
\begin{equation*}
C^+ = \{p + s(y - p): \; y \in X_2', \; s>0\},
\end{equation*}
and let $C^-$ denote its reflection about $p$, that is,
\begin{equation*}
C^- = \left\{ 2p - y : y \in C^+ \right\}.
\end{equation*}
Since $C^-$ is open and contains $z_2$, there exists an open neighbourhood $U$ of $z_2$ such that $U \subset C^-$ and $U \cap \Sigma = \varnothing$. It follows that for any $q \in U$, the set $(q; X_2']$ contains $p$ and hence an open neighbourhood of $p$. This contradicts our assumption, and consequently, for sufficiently small $\varepsilon_2 > 0$, \eqref{eq_l:cond} holds.

Since $z_2 \notin \Sigma$ and $\Sigma$ is closed, we may further shrink $\varepsilon_2$ if necessary to ensure that
\begin{equation*}
    \overline{B_{\varepsilon_2}(z_2)} \cap \Sigma = \varnothing, \qquad \{z_1\} \times B_{\varepsilon_2}(z_2) \subset S.
\end{equation*}

Finally, set $\varepsilon = \min\{\varepsilon_1, \varepsilon_2\}$ and define $Y_2 = B_\varepsilon(z_2)$, $X_2 = X_2'\setminus Y_2$. Then, properties \eqref{item:YinS}, \eqref{item:SigmaCovered}, and \eqref{item: Y2doesnotintersect} in the lemma are satisfied.

To verify property \eqref{item:existenceOfx1}, we take any $y = (z_1, y_2) \in Y$ and $x_2 \in X_2$. Then, by the triangle inequality,
\begin{equation*}
    |z_2 - x_2| - \varepsilon < |y_2 - x_2| < |z_2 - x_2| + \varepsilon.
\end{equation*}
Since $\varepsilon < \varepsilon_1$, it follows from the definition of $X_2$ that $B_\varepsilon(x_2) \subset W_2$. Hence, there exists $x_2' \in B_\varepsilon(x_2)$ on the line through $z_2$ and $x_2$ such that
\begin{equation}\label{eq_2_l:cond}
    |z_2 - x_2'| = |y_2 - x_2|.
\end{equation}
By the second part of Condition~\ref{condition_intro}, there exists $x_1 \in X_1$ such that
\begin{equation*}
    [z_1; x_1] \cap \Sigma = \varnothing \quad \text{and}   \quad |x_1 - z_1| = |z_2 - x_2'| \in (T-\ell, T+\ell).
\end{equation*}
This and \eqref{eq_2_l:cond} imply that property \eqref{item:existenceOfx1} holds as well.
\end{proof}

We now establish the following auxiliary lemma.

\begin{lemma}\label{cor:cond}
Assume that Condition \ref{condition_intro} holds. In particular, the conditions of Lemma \ref{l:cond} hold, and let $z, Y_2, Y$, $X_1$, and $X_2$ be the point and sets defined therein. Then, for every $y^* \in Y$, there exists a dense subset $X_2(y^*) \subset X_2$ such that for every $x_2 \in X_2(y^*)$, there exist $y \in S$, $x_1 \in X_1$, and $h \in \mathcal{C}_{\operatorname{sym}}(S)$ such that:
\begin{enumerate}
    \item \label{cor:item:y2inY_2} $y_2\in Y_2$.
    \item \label{item:cor-y2path} Either $y_2 \in [y_2^*; x_2]$ or $y_2^* \in [y_2; x_2]$.
    \item \label{item:cor-noSigma} $[y_1; x_1] \cap \Sigma = \varnothing$.
    \item \label{item:cor-equalDist} $|x_1 - y_1| = |x_2 - y_2| \in (T- \ell, T + \ell)$.
    \item \label{item:cor-nonzeroV} $v_0(t, x_1, x_2) \neq 0$, where $t = |x_1 - y_1|$ and $v = v^h$ is the solution to \eqref{main_eq}.
\end{enumerate}
\end{lemma}

\begin{proof}
    It suffices to prove that for each $y^*\in Y$, $x_2^*\in X_2$, and $\delta>0$, there exist $y\in S$, $x_1\in X_1$, $x_2\in X_2$, and $h \in \mathcal{C}_{\operatorname{sym}}(S)$ satisfying \eqref{cor:item:y2inY_2}-\eqref{item:cor-nonzeroV} and $|x_2^* - x_2|<\delta$. Indeed, in this case we may define $X_2(y^*)$ as the set of all such points $x_2 \in X_2$, which is then dense in $X_2$.

    Let $y^*\in Y$, $x_2^*\in X_2$, and let $x_1^* \in X_1$ be a point given by \eqref{item:existenceOfx1} in Lemma \ref{l:cond}. Let $\delta > 0$ be sufficiently small. Without loss of generality, we may assume that $y^*$ is the origin in $\mathbb{R}^{2n}$. Then we set $t^* = |x_1^*| = |x_2^*|$. Consider the set
\begin{equation*}
    U_\delta = (t^* - \delta, t^* + \delta) \times B_\delta(x_1^*) \times B_\delta(x_2^*),
\end{equation*}
where $B_\delta(x_j^*)$ denotes the open ball in $\mathbb{R}^n$ of radius $\delta$ centered at $x_j^*$. Let $\varepsilon \in (0, t^*/2)$ be sufficiently small, and assume that for every $h \in \mathcal{C}_{\operatorname{sym}}(S)$, we have $v_0^h\arrowvert_{U_\delta} = 0$.

Next, we construct a source for which the corresponding solution is not identically zero on the smaller set $U_{\delta/2}$. We will use coordinates given by \eqref{change_variables}. Let $\mu \in C_0^\infty(\mathbb{R})$ be a function supported in $(\varepsilon, 2\varepsilon)$ such that 
\begin{equation*}
    V(t,\omega_1,\omega_2) \int_{0}^{t}  \frac{\mu(s)}{V(s,\omega_1,\omega_2)}ds = 1,
\end{equation*}
where $V$ is the function defined by \eqref{A} and $\omega_j = x_j/|x_j|$. Let $\nu$ be a smooth function on $\mathbb{R}$ such that 
\begin{equation*}
    \nu (\sigma) = 
    \begin{cases}
        0 & \text{if } \sigma<\varepsilon,\\
        1/\sigma^l & \text{if } \sigma > 1,
    \end{cases}
\end{equation*}
for some large $l\in \mathbb{N}$. Let $F_0 \in I(K_\varepsilon)$ be a function supported in $(\varepsilon,2\varepsilon)\times S$ and symmetric with respect to the variables $x_1$ and $x_2$, with principal symbol given by 
\begin{equation*}
    \sigma[F_0](t,\omega_1,\omega_2, \sigma_1, \sigma_2) = \mu(t) \nu(\sigma_1) \nu(\sigma_2).  
\end{equation*}
Let $F= [F_0,0,0,0]^T$ and $u^F$ be the corresponding solution of \eqref{main_eq}. Due to  \eqref{princ_symb_u_in}, we know that 
\begin{equation*}
    \sigma[u_0^F] (t, \omega_1, \omega_2, \sigma_1, \sigma_2) = iV(t, \omega_1,\omega_2) \int_0^{t} \frac{\mu(s)}{V(s, \omega_1,\omega_2)}  \nu(\sigma_1) \nu(\sigma_2) ds = i\nu(\sigma_1) \nu(\sigma_2).
\end{equation*}
Since $l$ is large, it follows that $F\in C (\mathbb{R}; H^{k+1}(\mathbb{R}^{2n}; \mathbb{C}^4))$ for some $k\in \mathbb{N}$. By Proposition \ref{cor: stability_for_aux_eq}, $u^F \in C(\mathbb{R}; H^{k}(\mathbb{R}^{2n}; \mathbb{C}^4))$. Moreover, from the above expression for the principal symbol of $u_0^F$, we conclude that $u_0^F$ is not identically zero on $U_{\delta/2}$.

Let $\varkappa$ be a function appearing in the initial condition of \eqref{shifted_Av_eq_0}. We additionally assume that $\varkappa$ is symmetric with respect to the variables $x_1$ and $x_2$. Let $w$ and $u'$ be the solutions of \eqref{eq:for_omega} and \eqref{shifted_Av_eq_0}, respectively. By Lemma \ref{l:dif_smooth}, $u^F - u'$ is smooth, and hence, $u_0'$ is not identically zero on $U_{\delta/2}$.
 
Next, we set $f(x_1,x_2) = \varkappa(x_1,x_2) w(3\varepsilon,x_1,x_2)$ and let $u^f$ be the solution of \eqref{main_eq}. Then, $u^f(t,x_1,x_2) = u'(t+3\varepsilon,x_1,x_2)$, and therefore, $u_0^f$ is not identically zero on $U_{\delta}$.

On the other hand, we can approximate $f$ by a function $h \in \mathcal{C}_{\operatorname{sym}}(S)$ in the topology of $H^{k+1}(\mathbb{R}^{2n}; \mathbb{C}^4)$. By our assumption $v_0^h \arrowvert_{U_\delta} = 0$, and by Lemma \ref{direct_problem_est}, it follows that $u_0^f \arrowvert_{U_\delta} = 0$. This contradicts our earlier conclusion that $u_0^f$ is not identically zero on the large set $U_{\delta}$.
Therefore, there exist $(t, x_1, x_2) \in U_\delta$ and $h \in \mathcal{C}_{\operatorname{sym}}(S)$ such that
\begin{equation}\label{eq:v_noteq_zero}
v_0^h(t, x_1, x_2) \neq 0,
\end{equation}
where $v^h$ is the solution to equation \eqref{main_eq}. 

Next, we note that $|t - t^*| < \delta$. Therefore, since $\delta > 0$ is sufficiently small, it follows that $t \in (T - \ell, T + \ell)$. Moreover, as $Y_2$ is an open set, there exists a point $y_2$ on the line through $y_2^*$ and $x_2$ such that $|x_2 - y_2| = t$. Since $[y_1^*; x_1^*] \cap \Sigma = \varnothing$ and $\Sigma$ is closed, it follows that for sufficiently small $\delta > 0$, we have $[y_1^*; x_1] \cap \Sigma = \varnothing$, and consequently, there exists $y_1$ on the line through $y_1^*$ and $x_1$ such that
\begin{equation*}
    y=(y_1,y_2)\in S,
    \qquad
    [y_1,x_1]\cap\Sigma=\varnothing,
    \qquad
    |x_1 - y_1| = t.
\end{equation*}
By construction, the tuple $(y, x_1, x_2, h)$ satisfies all the required properties \eqref{cor:item:y2inY_2}-\eqref{item:cor-nonzeroV} and $|x_2^* - x_2|<\delta$.

\end{proof}

\begin{remark}\label{rem:choice_sources_osc_C}
Although the physical model is posed as an initial-value problem, we analyze the inverse problem via an auxiliary source-driven formulation. Instead of arbitrary initial data, we use time-localized sources with prescribed microlocal structure. Each such source can be matched with an initial condition so that the corresponding homogeneous and forced solutions agree modulo a smooth function for all $t>0$. As a result, the measurement data are unchanged, while the microlocal analysis becomes tractable.

The class of sources used in this work can be realized explicitly as oscillatory integrals: In polar coordinates $x_j = r_j \omega_j$ and with shifted radial variables $s_j = r_j - t$, a typical source has the form
$$
F_0(t,x_1,x_2)
 = \int e^{ i ( s_1 \sigma_1 + s_2 \sigma_2 ) }
   a(t,\omega_1,\omega_2,\sigma_1,\sigma_2)
   \, d\sigma_1 d\sigma_2,
$$
where the amplitude $a$ is smooth and compactly supported in $t$, and satisfies
$$
a(t,\omega_1,\omega_2,\sigma_1,\sigma_2)
 = \mu(t)\,(\Theta(\omega_1)+\Theta(\omega_2))\,
   \nu(\sigma_1)\,\nu(\sigma_2).
$$
The oscillatory phases concentrate the singularities of the source on the set $\{ |x_1| = |x_2| = t \}$. The angular profile $\Theta$ selects prescribed propagation directions, while the frequency cutoffs $\nu$ exclude a neighborhood of low frequencies where the operator $\mathcal{A}$ fails to behave pseudodifferentially. This structured choice of sources ensures that the leading singularities of the solution can be computed explicitly along characteristic rays, making the dependence on the unknown density $\rho$ transparent.
\end{remark}

Now, we are ready to prove the main result.

\begin{proof}[Proof of Theorem \ref{main result}]
Since Condition~\ref{condition_intro} holds, we may apply Lemma \ref{cor:cond}. Let $Y$ be the corresponding set, and fix any $y^* \in Y$. Let $X_2(y^*) \subset X_2$ be the dense subset provided by the corollary, and choose any $x_2 \in X_2(y^*)$ such that $[y_2^*,x_2]$ intersects $\Sigma$. Then, there exist $y \in S$, $x_1 \in X_1$, and $h \in \mathcal{C}_{\operatorname{sym}}(S)$ satisfying properties \eqref{cor:item:y2inY_2}-\eqref{item:cor-nonzeroV} of Lemma \ref{cor:cond}. Let $w=w^h$ be the corresponding solution to equation \eqref{main_eq}.

Without loss of generality, we assume that $y$ is the origin in $\mathbb{R}^{2n}$. Hence, by property~\eqref{item:cor-equalDist} of Lemma \ref{cor:cond}, we can write
\begin{equation*}
x_1 =  t' \kappa_1 \quad \text{and} \quad x_2 = t' \kappa_2,
\end{equation*}
where $\kappa_1 = x_1 / |x_1|$, $\kappa_2 = x_2 / |x_2|$, and $t' = |x_1| = |x_2|$. Since $[y_2^*,x_2]$ intersects $\Sigma$, properties \eqref{item: Y2doesnotintersect} of Lemma \ref{l:cond} and \eqref{item:cor-y2path} of Lemma \ref{cor:cond} imply that $[y_2,x_2]$ also intersects $\Sigma$, whereas \eqref{item:cor-noSigma} of Lemma \ref{cor:cond} ensures that $[y_1,x_1]\cap\Sigma=\varnothing$. As $y$ is the origin, this gives $x_1 \neq x_2$; and since $|x_1|=|x_2|$, we conclude that $\kappa_1 \neq \kappa_2$.

Recall that $Y = \{z_1\} \times Y_2$, where $z_1$ and $Y_2$ are as defined in Condition~\ref{condition_intro} and Lemma~\ref{l:cond}, respectively. Since $Y_2$ is open, we may choose a sufficiently small constant $\varepsilon > 0$ such that
\begin{equation*}
\{ s \kappa_2: \; 0 < s < 2 \varepsilon \} \subset Y_2.
\end{equation*}
Moreover, since $Y_2 \cap (\Sigma\cup X_2) = \varnothing$, it follows that $t' > \varepsilon$ and
\begin{equation}\label{eq:rho2_zero}
\rho(s \kappa_2) = 0 \quad \text{for all } s \in (0, 2\varepsilon).
\end{equation}
We define 
\begin{equation*}
    v(t,\cdot) = w(t - 3\varepsilon,\cdot), \qquad t>3\varepsilon,
\end{equation*}
and set $t^* = t'+3\varepsilon$. Then,
\begin{equation*}
    v_0(t^*, x_1,x_2) = w_0(t', x_1,x_2) \neq 0,
    \qquad
    \text{and}
    \qquad
    t^* \in (T- \ell + 3\varepsilon, T + \ell + 3\varepsilon)
\end{equation*}

Next, we construct a source in the same way as in Lemma~\ref{l: det_aux_funct}. We will use coordinates given by \eqref{change_variables}. Let $\mu \in C_0^\infty(\mathbb{R})$ be a function supported in $(\varepsilon, 2\varepsilon)$ such that  
\begin{equation*}
     V(t^*, \kappa_1, \kappa_2)\int_{\mathbb{R}}  \frac{\mu(s)}{V(s, \kappa_1, \kappa_2)} \, ds = 1,
\end{equation*}
where $V$ is the function given by \eqref{A}. Let $\Theta$ be a smooth function on $\mathbb{S}^{n-1}$, and let $\nu$ be the function defined by \eqref{nu}. Let $F_0 \in I(K_\varepsilon)$ be a function supported in $(\varepsilon,2\varepsilon)\times S$ and symmetric with respect to the variables $x_1$ and $x_2$, with principal symbol given by 
\begin{equation*}
    \sigma[F](t,\omega_1,\omega_2, \sigma_1, \sigma_2) = \mu(t) (\Theta(\omega_1) + \Theta(\omega_2) ) \nu(\sigma_1) \nu(\sigma_2).  
\end{equation*}
Let $F= [F_0,0,0,0]^T$ and $u = u^F$ be the corresponding solution to \eqref{auxiliary_eq}.

We note that $u_0^{in}$, defined in \eqref{u_in}, does not depend on the density $\rho$, and thus it can be computed. Therefore, by Corollary~\ref{cor:recover_v}, the given data determine the quantity $m(u_0^{in}, v_0)$ on the region
\begin{equation*}
    \mathcal{U} = (T - \ell + 3\varepsilon,\, T + \ell + 3\varepsilon) \times X_1.
\end{equation*}
Let $u'$ denotes the solution to equation \eqref{shifted_Av_eq_0}. By Lemma~\ref{l:dif_smooth}, the difference $u - u'$ is smooth on $(3\varepsilon, \infty) \times \mathbb{R}^{2n}$. Hence, by Lemma~\ref{polarization}, we can extract $m(u_0', v_0)$ from the data on the region $\mathcal{U}$. Therefore, we recover the full symbol of $m(u_0,v_0)$ restricted to the region $\mathcal{U}$. Combining these observations, we conclude that the data determine the full symbol of $m(u_0^{\mathrm{sc}}, v_0)$ on $\mathcal{U}$, where $u_0^{\mathrm{sc}}$ is defined in \eqref{u_sc}. Lemma~\ref{l: Pr_symb_u_in_sc} implies that

\begin{multline*}
    \sigma[u_0^{sc}] (t,\omega_1,\omega_2,\sigma_1,\sigma_2)\\
    = i g^2 V(t,\omega_1,\omega_2) \int_0^t  \frac{1}{V(r,\omega_1,\omega_2)} \left( \frac{\rho(r\omega_1)}{-\sigma_1} + \frac{\rho(r\omega_2)}{-\sigma_2}\right) \sigma[u_0^{in}] (r,\omega_1,\omega_2,\sigma_1,\sigma_2) dr
\end{multline*}
and
\begin{equation*}
    \sigma [u_0^{in}] (t, \omega_1, \omega_2, \sigma_1, \sigma_2)
    = i V(t,\omega_1,\omega_2) \int_0^t   \frac{\mu(s)}{V(s,\omega_1,\omega_2)}     (\Theta(\omega_1) + \Theta(\omega_2) ) \nu(\sigma_1) \nu(\sigma_2) ds
\end{equation*}
on $N_\varepsilon^*K_\varepsilon \cap \{t>3\varepsilon\}$. Therefore, we obtain
\begin{multline*}
    \sigma[u_0^{sc}] (t,\omega_1,\omega_2,\sigma_1,\sigma_2) =  - g^2 V(t,\omega_1,\omega_2) \\
     \times \int_0^t  \left( \frac{\rho(r\omega_1)}{-\sigma_1} + \frac{\rho(r\omega_2)}{-\sigma_2}\right) \int_{0}^r \frac{\mu(s)}{V(s,\omega_1,\omega_2)}  (\Theta(\omega_1) + \Theta(\omega_2) ) \nu(\sigma_1) \nu(\sigma_2) ds dr.
\end{multline*}
By Lemma~\ref{distr_R3}, we know that $m(u_0^{sc},v_0) \in I(K_\varepsilon^1)$, and we can express its principal symbol as follows:
\begin{multline*}
    \sigma[m(u_0^{sc},v_0)] (t,\omega_1,\sigma_1) =  g^2 \int_{\mathbb{R}} \int_{\mathbb{S}^{n-1}} e^{-i\sigma_2 t} [\mathcal{F}^{-1} \eta] (\sigma_2, \omega_2)  \overline{v_0(t, 0,\omega_1,0,\omega_2)}V(t,\omega_1,\omega_2)\\
    \times \int_0^t \left( \frac{\rho(r\omega_1)}{\sigma_1} + \frac{\rho(r\omega_2)}{\sigma_2}\right) \int_{0}^r \frac{\mu(s)}{V(s,\omega_1,\omega_2)}  (\Theta(\omega_1) + \Theta(\omega_2) )\nu(\sigma_1) \nu(\sigma_2) ds dr  d\omega_2d\sigma_2.
\end{multline*}
We evaluate the above at $\omega_1=\kappa_1$ and let $\Theta=\Theta_r \to \delta_{\kappa_2}$ as $r\to 0$, where $\delta_{\kappa_2}$ is the Dirac delta function concentrated at $\kappa_2$. Then, since $\Theta_r(\kappa_1) = 0$ as $r\to 0$, we obtain 
\begin{multline*}
    \lim_{r \rightarrow 0} \sigma[m(u_0^{sc},v_0)] (t,\kappa_1,\sigma_1) =  g^2 \int_{\mathbb{R}}  e^{-i\sigma_2 t} [\mathcal{F}^{-1} \eta] (\sigma_2, \kappa_2)  \overline{v_0(t, 0,\kappa_1,0,\kappa_2)} \\
    \times V(t,\kappa_1,\kappa_2)\int_0^t  \left( \frac{\rho(r\kappa_1)}{\sigma_1} + \frac{\rho(r\kappa_2)}{\sigma_2}\right) \int_{0}^r \frac{\mu(s)}{V(s,\kappa_1,\kappa_2)}   \nu(\sigma_1) \nu(\sigma_2) ds dr  d\sigma_2.
\end{multline*}
By property \eqref{item:cor-noSigma} of Lemma \ref{cor:cond}, we know that 
\begin{equation*}
    \rho(s\kappa_1) = 0, \quad \text{for all } s\in [0,t^*].
\end{equation*}
Therefore, by \eqref{eq:rho2_zero} and the definition of $\mu$, we obtain  
\begin{equation*}
    V(t,\kappa_1,\kappa_2)\int_0^t  \left( \frac{\rho(r\kappa_1)}{\sigma_1} + \frac{\rho(r\kappa_2)}{\sigma_2}\right) \int_{0}^r \frac{\mu(s)}{V(s,\kappa_1,\kappa_2)}  ds dr =  \int_0^t   \frac{\rho(r\kappa_2)}{\sigma_2} dr.
\end{equation*}
Hence, we conclude that
\begin{multline}\label{eq: recov_dens}
    \lim_{r \rightarrow 0} \sigma[m(u_0^{sc},v_0)] (t^*,\kappa_1,\sigma_1)\\ =  g^2  \nu(\sigma_1) \overline{v_0(t^*, 0,\kappa_1,0,\kappa_2)}   \int_{\mathbb{R}}  e^{-i\sigma_2 t^*} [\mathcal{F}^{-1} \eta] (\sigma_2, \kappa_2) \frac{\nu(\sigma_2)}{{\sigma_2}} d\sigma_2 
     \int_0^{t^*} \rho(r\kappa_2)    dr  .
\end{multline}
By the same reasoning as in the proof of Lemma~\ref{l: det_aux_funct}, we may assume without loss of generality that  
\begin{equation*}
     \int_{\mathbb{R}}  e^{-i\sigma_2 t^*} [\mathcal{F}^{-1} \eta] (\sigma_2, \kappa_2) \frac{\nu(\sigma_2)}{{\sigma_2}} d\sigma_2 \neq 0.
\end{equation*}
Using  property~\eqref{item:cor-nonzeroV} of Lemma \ref{cor:cond}, we know that
\begin{equation*}
v_0(t^*, 0, \kappa_1, 0, \kappa_2) \neq 0.
\end{equation*}
Moreover, $\nu (\sigma_1) \neq 0$ for $\sigma_1 > 1$. Hence, since the left-hand side of~\eqref{eq: recov_dens} is determined by the data, as previously established, and the functions $\nu$ and $\eta$ are known, we can recover the value of the integral
\begin{equation}\label{eq:det_rho}
\int_0^{t^*} \rho(r \kappa_2) dr .
\end{equation}

By Lemma \ref{l:cond}, we may assume that $Y_2$ is a ball. Since $y_2, y_2^* \in Y_2$, it follows that the segment $[y_2; y_2^*]$ is contained in $Y_2$. Hence $[y_2; y_2^*] \cap \Sigma = \varnothing$, and consequently, the integral of $\rho$ over the segment $[y_2; y_2^*]$ vanishes. Therefore, since we have determined \eqref{eq:det_rho}, and by property \eqref{item:cor-y2path} of Lemma \ref{cor:cond}, we can also recover the integral of $\rho$ over the segment $[y_2^*; x_2]$.

We summarize the above. For every $y^* \in Y = \{z_1\} \times Y_2$ and $x_2 \in X_2(y^*)$, the given data determine
\begin{equation*}
\int_0^1 \rho((1 - s)y_2^* + s x_2) \, ds.
\end{equation*}
Since $X_2(y^*) \subset X_2$ is dense and $\rho$ is smooth, the data determine the above integral for all $x_2 \in X_2$.

Finally, since this holds for every $y_2^* \in Y_2$, and $Y_2$ is open, and since $\Sigma \subset [y_2^*; X_2]$ by Lemma~\ref{l:cond}, we conclude from Theorem 3.6 from \cite{Natterer} that the given data determine the density $\rho$.
\end{proof}

\section*{Acknowledgments}

M.~L. was partially supported by the Advanced Grant project 101097198 of the European Research Council, Centre of Excellence of Research Council of Finland (grant 336786) and the FAME flagship of the Research Council of Finland (grant 359186). The second author has been partially funded by the Science Committee of the Ministry of Science and Higher Education of the Republic of Kazakhstan (Grant No. A22683207). L.O. and M.N. were supported by the European Research Council of the European Union, grant 101086697 (LoCal),
and the Research Council of Finland, grants 347715,
353096 (Centre of Excellence of Inverse Modelling and Imaging)
and 359182 (Flagship of Advanced Mathematics for Sensing Imaging and Modelling). The views and opinions expressed are those of the authors only and do not necessarily reflect those of the funding agencies or the EU.

\appendix
\section{}
In this appendix, we reprise the derivation of the system \eqref{eq:dynamics1}. The details are presented in \cite{KS2}.  

We consider a model for the interaction between a quantized massless scalar field and a system of two-level atoms. The atoms are taken to be identical, stationary and sufficiently well separated that interatomic interactions can be neglected. 
The overall system is described by the Hamiltonian
\begin{equation*}
\label{Htot}
H=H_F + H_A + H_I \ , 
\end{equation*}
where $H_F$ is the Hamiltonian of the field, $H_A$ is the Hamiltonian of the atoms and $H_I$ is the interaction Hamiltonian.
In order to treat the atoms and the field on the same footing, it is useful to introduce a real-space representation of $H$~\cite{KS1}.
The Hamiltonian of the field is of the form 
\begin{align}
\label{H_F}
 H_F = \int_{\R^n} dx (-\Delta)^{1/2}\phi^{\dagger}(x)\phi(x) ,
\end{align}
Here the operator $(-\Delta)^{1/2}$ is defined by the Fourier integral
\begin{align}
(-\Delta)^{1/2}f(x) = \int_{\R^n} e^{ik\cdot x} |k|  \mathcal{F}f(k) dk .
\end{align}
The field operators $\phi(x)$ and $\phi^\dagger(x)$, $x\in \R^n$, obey the bosonic commutation relations
\begin{align}
\label{commutation}
[\phi(x),\phi^{\dagger}(x')]&=\delta(x-x') \ , \\
[\phi(x),\phi(x')]&=0 .
\end{align}
Here $\phi^\dagger(x): F_{+}(L^2(\R^n))\to  F_{+}(L^2(\R^n))$ is an operator mapping the Fock space  $F_{+}(L^2(\R^n))$ 
(see formulas \eqref{Fock space total} and  \eqref{Fock space} below) that creates a photon at the point $x$ and $\phi(x):
F_{+}(L^2(\R^n)) \to  F_{+}(L^2(\R^n))$ annihilates a photon at $x$.
We note that (\ref{H_F}) is equivalent to the usual oscillator representation of $H_F$~\cite{KS1}.

The Hamiltonian of the atoms is given by
\begin{align*}
H_A = \Omega\int_{\R^n} dx\rho(x)\sigma^{\dagger}(x)\sigma(x) \ ,
\end{align*}
where $\Omega$ is the atomic resonance frequency, $\rho(x)$ is the number density of the atoms, and $\sigma$ 
is a Fermi field that obeys the anticommutation relations
\begin{align}
\label{anticommutation}
\{\sigma(x),\sigma^{\dagger}(x')\}&=\frac{1}{\rho(x)}\delta(x-x') \ , \\
\{\sigma(x),\sigma(x')\}&=0 \ .
\end{align}

The Hamiltonian describing the interaction between the field and the atoms, within the rotating wave approximation taken to be
 \begin{align*}
H_I = g\int_{\R^n} dx \rho(x)\left(\phi^\dagger(x)\sigma(x)+\sigma^{\dagger}(x)\phi(x)\right) \ ,
\end{align*}
where $g$ is the strength of the atom-field coupling.

The Hamiltonian $H$ acts on the tensor product of a symmetric and antisymmetric Fock space
\begin{align}\label{Fock space total}
   \mathcal{H}= F_{+}(L^2(\R^n))\otimes F_{-}(L^2(\R^n)) .
\end{align}
The Fock space $F_{\pm}(X)$ associated to the Hilbert space $X$ is defined as
\begin{align}\label{Fock space}
    F_{\pm} = \overline{\bigoplus_{n=0}^{\infty} S_{\pm} X^{\otimes n}},
\end{align}
where $S_{\pm}$ is the operator that symmetrizes (+) or antisymmetrizes (-) the $n$-fold tensor product $X^{\otimes n}$. Moreover, the operators $H_F$ and $H_I$ can be written as $H_F\otimes I$ and $I\otimes H_A$, respectively, where $I$ is the identity map.

We suppose that the system is in a two-excitation state of the form
\begin{align}
\label{state}
\vert\Psi\rangle = \int_{\R^n\times\R^n} d x_1 d x_2 & \left(\psi_2(x_1,x_2,t)\phi^{\dagger}(x_1)\phi^{\dagger}(x_2)
+ \psi_1(x_1,x_2,t)\rho(x_1)\sigma^{\dagger}(x_1)\phi^{\dagger}(x_2)\right.\\
+\nonumber&\left. a(x_1,x_2,t)\rho(x_1)\rho(x_2)\sigma^{\dagger}(x_1)\sigma^{\dagger}(x_2)\right)\vert 0\rangle \ ,
\end{align}
where $\vert 0\rangle$ is the combined vacuum state of the field and the ground state of the atoms. Here $\psi_2$ and $a$ are symmetric and antisymmetric, respectively:
\begin{equation*}
\psi_2(x_2,x_1,t) = \psi_2(x_1,x_2,t) \ , \quad a(x_2,x_1,t) = -a(x_1,x_2,t) \ ,
\end{equation*}
consistent with the bosonic  and fermionic character of the corresponding fields. We note that the field operator $\phi(x)$ annihilates the vacuum, so that $\phi(x)\vert 0\rangle=0$

The state $\vert\Psi\rangle$ is the most general two-photon state within the rotating wave approximation and evidently belongs to an invariant subspace of the Fock space. In addition, $\vert\Psi\rangle$ is normalized so that
\begin{align*}
\int_{\R^n\times\R^n} d x_1 d x_2 \left(2\vert\psi_2(x_1,x_2,t)\vert^2 + \rho(x_1)\vert \psi_1(x_1,x_2,t)\vert^2+2\rho(x_1)\rho(x_2)\vert a(x_1,x_2,t)\vert^2\right) = 1
\ ,
\end{align*}
which expresses the conservation of probability.

The dynamics of $\vert\Psi\rangle=\vert\Psi(t)\rangle$ is governed by the Schrodinger equation
\begin{align} 
\label{eq:schr}
i \partial_t\vert\Psi\rangle = H\vert\Psi\rangle \ .
\end{align}
Projecting onto the states $\phi^{\dagger}(\bx)\vert 0\rangle$ and $\sigma^{\dagger}(\bx)\vert 0\rangle$ and making use of (\ref{commutation}) and (\ref{anticommutation}), we arrive at 
the following system of equations obeyed by $a$, $\psi_1$ and
$\psi_2$:
\begin{align}
\label{eq:dynamics1again}
& i\partial_t\psi_2(\bx_1,\bx_2,t) =(-\Delta_{\bx_1})^{1/2}\psi_2(\bx_1,\bx_2,t)+(-\Delta_{\bx_2})^{1/2}\psi_2(\bx_1,\bx_2,t) \\
\nonumber
\label{eq:dynamics2}
&+\frac{g}{2}(\rho(\bx_1)\psi_1(\bx_1,\bx_2,t)+\rho(\bx_2)\psi_1(\bx_2,\bx_1,t)) \ , \\
 & \rho(\bx_1)i\partial_t \psi_1(\bx_1,\bx_2,t) = 2\coup \rho(\bx_1)\psi_2(\bx_1,\bx_2,t)+\rho(\bx_1)\left[(-\Delta_{\bx_2})^{1/2}+\Omega\right]\psi_1(\bx_1,\bx_2,t)\\
 \nonumber
 &-2\coup\rho(\bx_1)\rho(\bx_2)a(\bx_1,\bx_2,t) \ , \\
  &  \rho(\bx_1)\rho(\bx_2)i\partial_t a(\bx_1,\bx_2,t) =\rho(\bx_1)\rho(\bx_2)\frac{g}{2}(\psi_1(\bx_2,\bx_1,t)-\psi_1(\bx_1,\bx_2,t)) + 2\rho(\bx_1)\rho(\bx_2)\Omega a(\bx_1,\bx_2,t) \ ,
\label{eq:dynamics3}
\end{align}
which corresponds to \eqref{eq:dynamics1}.
The derivation of the above  system proceeds as follows.
We begin by evaluating both sides of the Schrodinger equation  (\ref{eq:schr}). The left-hand side is given by
\begin{align}
 \nonumber   i \partial_t\vert\Psi\rangle = \int_{\R^n\times\R^n} dx_1 d x_2 & \left(i\partial_t\psi_2(\bx_1,\bx_2,t)\phi^{\dagger}(\bx_1)\phi^{\dagger}(\bx_2)
+ i\partial_t\psi_1(\bx_1,\bx_2,t)\rho(\bx_1)\sigma^{\dagger}(\bx_1)\phi^{\dagger}(\bx_2)\right.\\
+&\left. i \partial_ta(\bx_1,\bx_2,t)\rho(\bx_1)\rho(\bx_2)\sigma^{\dagger}(\bx_1)\sigma^{\dagger}(\bx_2)\right)\vert 0\rangle \ ,
\end{align}
and the right-hand side is given by
\begin{align}
H\vert\Psi\rangle & =\int_{\R^n\times\R^n\times\R^n} dx d x_1 d x_2\left\{(-\Delta)^{1/2}\phi^{\dagger}(\bx)\phi(\bx) +\Omega\rho(\bx)\sigma^{\dagger}(\bx)\sigma(\bx) + \coup\rho(\bx)\left(\phi^{\dagger}(\bx)\sigma(\bx)+\phi(\bx)\sigma^{\dagger}(\bx) \right)\right\} \\
   \nonumber &\times\{\psi_2(\bx_1,\bx_2,t)\phi^{\dagger}(\bx_1)\phi^{\dagger}(\bx_2)
+ \psi_1(\bx_1,\bx_2,t)\rho(\bx_1)\sigma^{\dagger}(\bx_1)\phi^{\dagger}(\bx_2) \\
\nonumber&+ a(\bx_1,\bx_2,t)\rho(\bx_1)\rho(\bx_2)\sigma^{\dagger}(\bx_1)\sigma^{\dagger}(\bx_2)\}\vert 0\rangle\\
\nonumber&=\int _{\R^n\times\R^n\times\R^n}dx d x_1 d x_2 \{ \psi_2(\bx_1,\bx_2,t)(-\Delta)^{1/2}\phi^{\dagger}(\bx)\phi(\bx)\phi^{\dagger}(\bx_1)\phi^{\dagger}(\bx_2)\\
\nonumber&+\psi_1(\bx_1,\bx_2,t)\rho(\bx_1)(-\Delta)^{1/2}\phi^{\dagger}(\bx)\phi(\bx)\sigma^{\dagger}(\bx_1)\phi^{\dagger}(\bx_2)\\
\nonumber&+\psi_1(\bx_1,\bx_2,t)\rho(\bx_1)\Omega\rho(\bx)\sigma^{\dagger}(\bx)\sigma(\bx)\sigma^{\dagger}(\bx_1)\phi^{\dagger}(\bx_2)\\
\nonumber&+a(\bx_1,\bx_2,t)\rho(\bx_1)\rho(\bx_2)\Omega\rho(\bx)\sigma^{\dagger}(\bx)\sigma(\bx)\sigma^{\dagger}(\bx_1)\sigma^{\dagger}(\bx_2)\\
\nonumber\nonumber&+\psi_1(\bx_1,\bx_2,t)\rho(\bx_1)\coup\rho(\bx)\phi^{\dagger}(\bx)\sigma(\bx)\sigma^{\dagger}(\bx_1)\phi^{\dagger}(\bx_2)\\
\nonumber&+a(\bx_1,\bx_2,t)\rho(\bx_1)\rho(\bx_2)\coup\rho(\bx)\phi^{\dagger}(\bx)\sigma(\bx)\sigma^{\dagger}(\bx_1)\sigma^{\dagger}(\bx_2)\\
\nonumber&+\psi_1(\bx_1,\bx_2,t)\rho(\bx_1)\coup\rho(\bx)\phi(\bx)\sigma^{\dagger}(\bx)\sigma^{\dagger}(\bx_1)\phi^{\dagger}(\bx_2)\\
\nonumber&+\psi_2(\bx_1,\bx_2,t)\coup\rho(\bx)\phi(\bx)\sigma^{\dagger}(\bx)\phi^{\dagger}(\bx_1)\phi^{\dagger}(\bx_2)\}\vert 0\rangle.
\end{align}
Using the commutation relations (\ref{commutation}) twice, together with $\phi(x) \vert 0\rangle = 0$, we rewrite the first term on the right-hand side
\begin{align*}   
&\int _{\R^n\times\R^n\times\R^n}dx d x_1 d x_2 \psi_2(\bx_1,\bx_2,t)(-\Delta)^{1/2}\phi^{\dagger}(\bx)\phi(\bx)\phi^{\dagger}(\bx_1)\phi^{\dagger}(\bx_2)\vert 0\rangle
\\&\quad=
\int _{\R^n\times\R^n\times\R^n}dx d x_1 d x_2 \psi_2(\bx_1,\bx_2,t)(-\Delta)^{1/2}\phi^{\dagger}(\bx)(\phi^{\dagger}(\bx_1)\phi(\bx) + \delta(\bx - \bx_1))\phi^{\dagger}(\bx_2)\vert 0\rangle
\\&\quad=
\int _{\R^n\times\R^n\times\R^n}dx d x_1 d x_2 \psi_2(\bx_1,\bx_2,t)(-\Delta)^{1/2}\phi^{\dagger}(\bx)\phi^{\dagger}(\bx_1)\phi(\bx)\phi^{\dagger}(\bx_2)\vert 0\rangle
\\&\qquad+
\int _{\R^n\times\R^n}d x_1 d x_2 \psi_2(\bx_1,\bx_2,t)(-\Delta)^{1/2}\phi^{\dagger}(\bx_1)\phi^{\dagger}(\bx_2)\vert 0\rangle
\\&\quad=
\int _{\R^n\times\R^n\times\R^n}dx d x_1 d x_2 \psi_2(\bx_1,\bx_2,t)(-\Delta)^{1/2}\phi^{\dagger}(\bx)\phi^{\dagger}(\bx_1)\phi(\bx)\phi^{\dagger}(\bx_2)\vert 0\rangle
\\&\qquad+
\int _{\R^n\times\R^n}d x_1 d x_2 (-\Delta)^{1/2}_{\bx_1}\psi_2(\bx_1,\bx_2,t)\phi^{\dagger}(\bx_1)\phi^{\dagger}(\bx_2)\vert 0\rangle
\\&\quad=
\int _{\R^n\times\R^n}d x_1 d x_2 (-\Delta)^{1/2}_{\bx_2}\psi_2(\bx_1,\bx_2,t)\phi^{\dagger}(\bx_1)\phi^{\dagger}(\bx_2)\vert 0\rangle
\\&\qquad+
\int _{\R^n\times\R^n}d x_1 d x_2 (-\Delta)^{1/2}_{\bx_1}\psi_2(\bx_1,\bx_2,t)\phi^{\dagger}(\bx_1)\phi^{\dagger}(\bx_2)\vert 0\rangle.
\end{align*}
Recalling the anticommutation relations (\ref{anticommutation}) as well,
we can treat the other terms in a similar manner and arrive at
\begin{align*}
\nonumber&\int_{\R^n\times\R^n} d x_1 d x_2 \{ ((-\Delta_{\bx_1})^{1/2}\psi_2(\bx_1,\bx_2,t)+(-\Delta_{\bx_2})^{1/2}\psi_2(\bx_1,\bx_2,t))\phi^{\dagger}(\bx_1)\phi^{\dagger}(\bx_2)\\
\nonumber&+(-\Delta_{\bx_2})^{1/2}\psi_1(\bx_1,\bx_2,t)\rho(\bx_1)\sigma^{\dagger}(\bx_1)\phi^{\dagger}(\bx_2)+\psi_1(\bx_1,\bx_2,t)\rho(\bx_1)\Omega\sigma^{\dagger}(\bx_1)\phi^{\dagger}(\bx_2)\\
\nonumber&+2 a(\bx_1,\bx_2,t)\rho(\bx_1)\rho(\bx_2)\Omega\sigma^{\dagger}(\bx_1)\sigma^{\dagger}(\bx_2)+\psi_1(\bx_1,\bx_2,t)\rho(\bx_1)\coup\phi^{\dagger}(\bx_1)\phi^{\dagger}(\bx_2)\\
\nonumber&-2\coup a(\bx_1,\bx_2,t)\rho(\bx_1)\rho(\bx_2)\phi^{\dagger}(\bx_2)\sigma^{\dagger}(\bx_1)+
\coup\psi_1(\bx_1,\bx_2,t)\rho(\bx_1)\sigma^{\dagger}(\bx_1)\sigma^{\dagger}(\bx_2)\\
&+2\coup\psi_2(\bx_1,\bx_2,t)\rho(\bx_1)\sigma^{\dagger}(\bx_1)\phi^{\dagger}(\bx_2)\}\vert 0\rangle.
\end{align*}
Computing the inner products 
\begin{align*}
   \langle 0\vert\phi(\bx_1)\phi(\bx_2)i\partial_t\vert\Psi\rangle &= 2i\partial_t\psi_2(\bx_1,\bx_2,t) \ ,\\
    \nonumber\langle 0\vert\phi(\bx_1)\phi(\bx_2)H\vert\Psi\rangle &= 2((-\Delta_{\bx_1})^{1/2}+(-\Delta_{\bx_2})^{1/2})\psi_2(\bx_1,\bx_2,t)\\
    &+\coup\rho(\bx_1)\psi_1(\bx_1,\bx_2,t)+\coup\rho(\bx_2)\psi_1(\bx_2,\bx_1,t)\, , 
\end{align*}
yields (\ref{eq:dynamics1again}). In addition, we have
\begin{align*}
     \langle 0\vert\phi(\bx_2)\sigma(\bx_1)\rho(\bx_1)i\partial_t\vert\Psi\rangle &= \rho(\bx_1)i\partial_t\psi_1(\bx_1,\bx_2,t) \ , \\
    \nonumber\langle 0\vert\phi(\bx_2)\sigma(\bx_1)\rho(\bx_1)H\vert\Psi\rangle &= ((-\Delta_{\bx_2})^{1/2}+\Omega)\rho(\bx_1)\psi_1(\bx_1,\bx_2,t)\\
    &+2\coup\rho(\bx_1)\psi_2(\bx_1,\bx_2,t)-2\coup\rho(\bx_1)\rho(\bx_2)a(\bx_1,\bx_2,t)\, , 
\end{align*}
which gives (\ref{eq:dynamics2}). Finally,
\begin{align*}
    \langle 0\vert\sigma(\bx_1)\sigma(\bx_2)\rho(\bx_1)\rho(\bx_2)i\partial_t\vert\Psi\rangle &= 2\rho(\bx_1)\rho(\bx_2)i\partial_t a(\bx_1,\bx_2,t) \ , \\
    \nonumber\langle 0\vert\sigma(\bx_1)\sigma(\bx_2)\rho(\bx_1)\rho(\bx_2)H\vert\Psi\rangle &= \coup\rho(\bx_1)\rho(\bx_2)\psi_1(\bx_1,\bx_2,t)-\coup\rho(\bx_1)\rho(\bx_2)\psi_1(\bx_2,\bx_1,t)\\
    &+4\Omega\rho(\bx_1)\rho(\bx_2)a(\bx_1,\bx_2,t)\  ,
\end{align*}
gives (\ref{eq:dynamics3}).

\section{}\label{appendix: source problem}
In this appendix, we show how the methods developed for the initial value problem can be adapted to a related inverse problem where a source term appears on the right-hand side of the main equation. Specifically, we consider the system
\begin{equation}\label{second_IP_equation}
    \begin{cases}
        \mathcal{A}u = F, \\
        u\arrowvert_{t<0}  = 0,
    \end{cases}
\end{equation}
where the source $F$ belongs to the set
\begin{equation*}
	\mathcal{C}_{\operatorname{sym}}^t(S)= \left\{ F \in C_0^\infty((0,\infty)\times S; \mathbb{C}^4) : F_0(t,x_1, x_2) = F_0(t,x_2, x_1), \;  F_1 = F_2 = F_3 = 0 \right\}.
\end{equation*}

By Proposition \ref{cor: stability_for_aux_eq}, the following measurement map is well-defined:
\begin{align*}
	\Gamma :   \mathcal{C}_{\operatorname{sym}}^t(S) &\to C^\infty ((0,\infty)\times W_1),\\
	\Gamma F(t,x_1) & = \int_{W_2} \chi(x_2) \left|u_0^F(t, x_1, x_2)\right|^2 \, dx_2,
\end{align*}
where $u_0^F$ denotes the first component of the solution $u^F$ to the system above.

The inverse problem is to determine the unknown density $\rho$ from knowledge of the map $\Gamma$.

To solve this problem, we require a geometric assumption analogous to Condition~\ref{condition_intro}. However, in contrast to the initial value problem, the synchronization condition can be relaxed. While we still assume the existence of a source point that connects the measurement points $(x_1, x_2) \in W_1 \times W_2$ via rays, we no longer require that the distances from this point to $x_1$ and $x_2$ are equal to a fixed value $t$. This relaxation is possible because the source $F$ now depends on time, allowing us to vary the emission time so that the two photons can still arrive at $x_1$ and $x_2$ simultaneously. Namely, we require the following assumptions:
\begin{condition}\label{condition_relaxed}
	There exists a set $X_1 \subset W_1$ such that for every $x_1 \in X_1$ and $x_2 \in W_2$, there exists a point $y = (y_1, y_2) \in S$
satisfying
\begin{equation*}
    |x_1 - y_1| = |x_2 - y_2|>0.
\end{equation*}
Additionally, there exists a point $z = (z_1, z_2) \in S$ such that:
\begin{enumerate}
    \item $\Sigma \subset (z_2; W_2]$.
    \item For every $x_2 \in W_2$, there exists a point $x_1 \in X_1$ such that:
    \begin{itemize}
        \item[a)] The segment $[z_1; x_1]$ does not intersect $\Sigma$, that is, $[z_1; x_1] \cap \Sigma = \varnothing$.
        \item[b)] The distances satisfy $|x_1 - z_1| = |x_2 - z_2|$.
    \end{itemize}
\end{enumerate}
\end{condition}

The proof of the following lemma is the same as the proof of Lemma \ref{polarization}.
\begin{lemma}\label{polarization_second_IP}
    Let $F$, $H\in\mathcal{C}_{\operatorname{sym}}^t(S)$ and $u^F$, $u^H$ be the solutions of \eqref{second_IP_equation}. Then, the map $\Gamma$ determines $m(u_0^F,u_0^H)$ on $(0,\infty)\times W_1$, where $m(\cdot,\cdot)$ is defined in \eqref{def_form}.
\end{lemma}

Next, we prove an analogue of Corollary \ref{cor:recover_v}:
\begin{lemma}\label{rec_v_second_IP}
    Assume that Condition~\ref{condition_relaxed} holds. Let $H \in \mathcal{C}_{\operatorname{sym}}^t(S)$, and let $v = v^H$ be the solution to \eqref{second_IP_equation}. Then, for distinct points $x_1^* \in X_1$, $x_2^* \in W_2$, the measurement map $\Gamma$ determines $v_0(\cdot, x_1^*, x_2^*)$ on $(0,\infty)$.
\end{lemma}

\begin{proof}
    By Condition \ref{condition_relaxed}, there exists a point $y = (y_1,y_2) \in S$ such that
    \begin{equation*}
    |x_1^* - y_1| = |x_2^* - y_2| = r>0.
\end{equation*}
Without loss of generality, we assume that $y$ is the origin in $\mathbb{R}^{2n}$; otherwise, we construct $K_\varepsilon$ originating from $y$ and repeat the preceding arguments. By assumption, the origin belongs to $S$, and we have $|x_1^*| = |x_2^*| = r$.

Let $t^*>0$. If $t^* =r$, we consider the source $F$ constructed in the proof of Lemma~\ref{l: det_aux_funct} and let $u=u^F$ be the solution to \eqref{second_IP_equation}. Repeating the steps of Lemma~\ref{l: det_aux_funct}, we obtain
\begin{equation*}
    \lim_{\epsilon\rightarrow 0}\sigma[m(u_0,v_0)](t^*,\omega_1^*,\sigma_1) = \overline{v_0(t^*, 0,\omega_1^*,0,\omega_2^*)} \nu(\sigma_1) \int_{\mathbb{R}} e^{-i\sigma_2 t^*} [\mathcal{F}^{-1} \eta] (\sigma_2, \omega_2^*) \nu(\sigma_2)  d\sigma_2.
\end{equation*}
As in Lemma \ref{l: det_aux_funct}, this allows us to determine the value of $v_0(t^*, 0,\omega_1^*,0,\omega_2^*)$.

Now, suppose $t^* > r$. Since the set $K_\varepsilon$ can be translated forward in time by $t^* - r$, Lemma~\ref{l: det_aux_funct} together with Remark~\ref{remak_shift} implies that the data determine $v_0(t^*, x_1^*, x_2^*)$ in this case as well.

    Finally, consider the case $t^* < r$. Define a shifted source $G$ by
    \begin{equation*}
        G(t, \cdot, \cdot) = H(t - r + t^*, \cdot, \cdot).
    \end{equation*}
    In particular,
    \begin{equation*}
        G(0, \cdot, \cdot) = H(-r + t^*, \cdot, \cdot).
    \end{equation*}
    Since $t^* < r$, it follows that $G \in \mathcal{C}_{\operatorname{sym}}^t(S)$. Letting $w = w^G$ denote the corresponding solution to \eqref{second_IP_equation}, we observe that
    \begin{equation*}
        w(t, \cdot, \cdot) = v(t - r + t^*, \cdot, \cdot).
    \end{equation*}
    Using the result from the case $t^* = r$ and the assumption $|x_1^*| = |x_2^*| = r$, we recover the value $w_0(r, x_1^*, x_2^*)$. By the relation between $w$ and $v$, this yields $v_0(t^*, x_1^*, x_2^*)$.
\end{proof}

The following two results can be established by arguments identical to those used in the proofs of Lemmas~\ref{l:cond} and \ref{cor:cond}.

\begin{lemma}\label{l:cond_second_IP}
Assume that Condition~\ref{condition_relaxed} holds, and let $z$ and $X_1$ be the point and the set given there. Then there exist a neighbourhood $Y_2$ of $z_2$ and an open subset $X_2 \subset W_2$ such that:
\begin{enumerate}
    \item \label{item:YinS_second_IP} $Y = \{z_1\} \times Y_2 \subset S$.
    \item \label{item:SigmaCovered_second_IP} For every $p \in Y_2$, we have $\Sigma \subset (p; X_2]$.
    \item \label{item:existenceOfx1_second_IP} For every $y \in Y$ and $x_2 \in X_2$, there exists a point $x_1 \in X_1$ such that:
    \begin{itemize}
        \item[a)] \label{item:noIntersection_second_IP} The segment $[y_1; x_1]$ does not intersect $\Sigma$, that is, $[y_1; x_1] \cap \Sigma = \varnothing$.
        \item[b)] \label{item:distanceEquality_second_IP} The distances satisfy $|x_1 - y_1| = |x_2 - y_2|$.
    \end{itemize}
\end{enumerate}
Moreover, we can choose sets $Y_2$ and $X_2$ so that $Y_2$ is a ball in $\mathbb{R}^n$ and $Y_2\cap X_2=\varnothing$.
\end{lemma}

Next lemma is an analogue of Lemma~\ref{cor:cond}.

\begin{lemma}\label{cor:cond_second_IP}
Assume that Condition~\ref{condition_relaxed} holds. In particular, the conditions of Lemma~\ref{l:cond_second_IP} hold, and let $z, Y_2, Y$, $X_1$, and $X_2$ be the point and sets defined therein. Then, for every $y^* \in Y$, there exists a dense subset $X_2(y^*) \subset X_2$ such that for every $x_2 \in X_2(y^*)$, there exist $y \in S$, $x_1 \in X_1$, and a source $H \in \mathcal{C}_{\operatorname{sym}}^t(S)$ such that:
\begin{enumerate}
    \item \label{cor:item:y2inY_2_second_IP} $y_2\in Y_2$.
    \item \label{item:cor-y2path_second_IP} Either $y_2 \in [y_2^*; x_2]$ or $y_2^* \in [y_2; x_2]$.
    \item \label{item:cor-noSigma_second_IP} $[y_1; x_1] \cap \Sigma = \varnothing$.
    \item \label{item:cor-equalDist_second_IP} $|x_1 - y_1| = |x_2 - y_2|$.
    \item \label{item:cor-nonzeroV_second_IP} $v_0(t, x_1, x_2) \neq 0$, where $t = |x_1 - y_1|$ and $v = v^H$ is the solution to \eqref{second_IP_equation}.
\end{enumerate}
\end{lemma}

Finally, we will show that $\Gamma$ determines the density.
\begin{theorem}
\label{main result_second_IP}
	Assume that Condition \ref{condition_relaxed} holds. Then the measurement map $\Gamma$ determines the density $\rho$ uniquely.
\end{theorem}

\begin{proof}
    The core of the argument is already present in the proof of Theorem~\ref{main result}, so we provide only a sketch here. Since Condition \ref{condition_relaxed} holds, we may apply Lemma \ref{cor:cond_second_IP}. Let $Y$ be the corresponding set, and fix any $y^* \in Y$. Let $X_2(y^*) \subset X_2$ be the dense subset provided by the corollary, and choose any $x_2 \in X_2(y^*)$ such that $[y_2^*,x_2]$ intersects $\Sigma$. Then, there exist $y \in S$, $x_1 \in X_1$, and a source $H \in \mathcal{C}_{\operatorname{sym}}^t(S)$ satisfying properties \eqref{cor:item:y2inY_2_second_IP}-\eqref{item:cor-nonzeroV_second_IP} of Lemma \ref{cor:cond_second_IP}. Let $v = v^H$ be the corresponding solution to equation \eqref{second_IP_equation}.

    Without loss of generality, we may assume that $y$ is the origin in $\mathbb{R}^{2n}$. By property~\eqref{item:cor-equalDist_second_IP}, we then have
    \begin{equation*}
        x_1 = t^* \kappa_1 \quad \text{and} \quad x_2 = t^* \kappa_2,
    \end{equation*}
    where $\kappa_1 = x_1 / |x_1|$, $\kappa_2 = x_2 / |x_2|$, and $t^* = |x_1| = |x_2|$. By the final statement of Lemma~\ref{l:cond_second_IP}, we have $t^* > 0$, and due to the choice of $x_2$, it follows that $\kappa_1 \neq \kappa_2$.

    Let $F$ be the source constructed in the proof of Theorem~\ref{main result}, and let $u = u^F$ be the corresponding solution to equation~\eqref{second_IP_equation}.

    Since $u_0^{in}$, defined in~\eqref{u_in}, does not depend on the density $\rho$, it can be computed. By Lemma~\ref{rec_v_second_IP}, the map $\Gamma$ determines $m(u_0^{in}, v_0)$ on $(0,\infty)\times W_1$. Lemma~\ref{polarization_second_IP} then implies that we can extract $m(u_0, v_0)$ from the data, and hence we recover $m(u_0^{sc}, v_0)$ on $(0,\infty)\times W_1$, where $u_0^{sc}$ is defined by~\eqref{u_sc}.

    In the proof of Theorem~\ref{main result}, it was shown that
    \begin{multline*}
        \lim_{\epsilon \rightarrow 0} \sigma[m(u_0^{sc},v_0)] (t^*,\kappa_1,\sigma_1) \\
        = \nu(\sigma_1) \overline{v_0(t^*, 0,\kappa_1,0,\kappa_2)}   
        \int_{\mathbb{R}}  e^{-i\sigma_2 t^*} [\mathcal{F}^{-1} \eta] (\sigma_2, \kappa_2) 
        \frac{\nu(\sigma_2)}{\sigma_2} \, d\sigma_2 
        \int_0^{t^*} \rho(r\kappa_2) \, dr.
    \end{multline*}
    As shown previously, this allows us to determine the value
    \begin{equation*}
        \int_0^1 \rho((1 - s)y_2^* + s x_2) \, ds
    \end{equation*}
    for every $y^* \in Y = \{z_1\} \times Y_2$ and $x_2 \in X_2(y^*)$. This suffices to recover $\rho$.
\end{proof}

\bibliographystyle{plain}
\bibliography{references}

\end{document}